\documentclass{compositio}
\usepackage{amscd,amssymb, amsmath,amsfonts, wasysym, mathrsfs, mathtools,enumerate,hhline,color}
\usepackage{graphicx}
\usepackage[all, cmtip]{xy}

\usepackage{url}

\definecolor{hot}{RGB}{65,105,225}

\usepackage[pagebackref=true,colorlinks=true, linkcolor=hot ,  citecolor=hot, urlcolor=hot]{hyperref}

\theoremstyle{plain}
\newtheorem{theorem}{Theorem}[section]
\newtheorem{prop}[theorem]{Proposition}
\newtheorem{lm}[theorem]{Lemma}

\newtheorem{cor}[theorem]{Corollary}

\newtheorem{lemma}[theorem]{Lemma}
\newtheorem{thrm}[theorem]{Theorem}

\theoremstyle{definition}

\newtheorem{defn}[theorem]{Definition}

\theoremstyle{remark}
\newtheorem{rmk}[theorem]{Remark}

\newtheorem*{assumption}{Assumption}

\newtheorem*{ex*}{Example}

\newtheorem{que}[theorem]{Question}
\newtheorem{defprop}[theorem]{Definition-Proposition}

\newcommand\sA{{\mathcal A}}
\newcommand\sO{{\mathcal O}}

\newcommand\sC{{\mathscr C}}

\newcommand\sF{{\mathcal F}}
\newcommand\sE{{\mathcal E}}

\newcommand\sM{{\mathcal M}}

\newcommand\sV{{\mathcal V}}
\newcommand\sL{\mathcal{L}}
\newcommand\sR{\mathcal{R}}

\def\sQ{\mathcal{Q}}
\def\bR{\mathbf{R}}

\newcommand\pp{{\mathbb{P}}}

\newcommand\zz{{\mathbb{Z}}}

\newcommand\cc{{\mathbb{C}}}

\newcommand\nn{{\mathbb{N}}}

\newcommand\hh{{\mathbb{H}}}

\newcommand{\mb}{\mathcal{M}_{\textrm{B}}}
\newcommand{\mdr}{\mathcal{M}_{\textrm{DR}}}
\newcommand{\mdol}{\mathcal{M_{\textrm{Dol}}}}

\newcommand{\Art}{\mathsf{ART}}
\newcommand{\Set}{\mathsf{SET}}

\DeclareMathOperator{\codim}{codim}              
\DeclareMathOperator{\id}{id}                    
\DeclareMathOperator{\obj}{Obj}
\DeclareMathOperator{\ad}{ad}
\DeclareMathOperator{\morph}{Morph}

\DeclareMathOperator{\iso}{Iso}
\DeclareMathOperator{\homo}{Hom}
\DeclareMathOperator{\enmo}{\mathcal{E}\hspace{-2pt}\it{nd}}
\DeclareMathOperator{\spec}{Spec}

\DeclareMathOperator{\rank}{Rank}

\def\ra{\rightarrow}

\def\bone{\mathbf{1}}
\def\bC{\mathbb{C}}
\def\cL{{\mathcal L}}
\def\cM{\mathcal{M}}
\def\Def{{\rm {Def}}}
\def\cR{\mathcal{R}}
\def\om{\omega}

\def\End{{\rm {End}}}
\def\lam{\lambda}
\def\lra{\longrightarrow}
\def\bP{\mathbb{P}}
\def\cO{\mathcal{O}}
\def\Pic{{\rm{Pic}}}
\def\cK{\mathcal{K}}
\def\cF{\mathcal{F}}
\def\Tor{{\rm{Tor}}}
\def\cV{\mathcal{V}}
\def\cA{\mathcal A}
\def\ti{\tilde}

\newcommand{\ubul}{{\,\begin{picture}(-1,1)(-1,-3)\circle*{2}\end{picture}\ }}

\newcommand{\cbul}{{\,\begin{picture}(-1,1)(-1,-3)\circle*{2}\end{picture}\ }}

\newcommand{\sAp}{\mathcal{A}^{p,\ubul}_{\rm{Dol}}}

\newcommand{\sAr}{\mathcal{A}^{\ubul}_{\rm{DR}}}

\newcommand{\sAg}{\mathcal{A}^{\ubul}_{\rm{Gysin}}}

\newcommand{\Ap}{{{A}}^{p,\ubul}_{\rm{Dol}}}
\newcommand{\Ad}{{A}^{0,\ubul}_{\rm{Dol}}}
\newcommand{\Ar}{{A}^{\ubul}_{\rm{DR}}}
\newcommand{\Ah}{{A}^{\ubul}_{\rm{Higgs}}}
\newcommand{\Ag}{{A}^{\ubul}_{\rm{Gysin}}}

\title[Cohomology Jump Loci of DGLA]{Cohomology Jump Loci of Differential Graded Lie Algebras}
\begin{document}
\author{Nero Budur}
\email{Nero.Budur@wis.kuleuven.be}
\address{KU Leuven and University of Notre Dame}
\curraddr{KU Leuven, Department of Mathematics,
Celestijnenlaan 200B, B-3001 Leuven, Belgium}

\author{Botong Wang}
\email{bwang3@nd.edu}
\address{University of Notre Dame}
\curraddr{Department of Mathematics,
 255 Hurley Hall, IN 46556, USA}

\date{}
\thanks{The first author was partly sponsored by the Simons Foundation, NSA, and a KU Leuven OT grant. }
\begin{abstract} 
To study infinitesimal deformation problems with cohomology constraints, we introduce and study cohomology jump functors for differential graded Lie algebra (DGLA) pairs. We apply this to local systems, vector bundles, Higgs bundles, and representations of fundamental groups. The results obtained describe the analytic germs of the cohomology jump loci inside the corresponding moduli space, extending previous results of Goldman-Millson, Green-Lazarsfeld, Nadel, Simpson, Dimca-Papadima, and of the second author.
\end{abstract}

\maketitle

\section{Introduction}

\subsection{Motivation and overview.}\label{subsMot}  Consider a representation $\rho:\pi_1(X)\ra GL(n,\bC)$ of the fundamental group of a topological space $X$. How can one describe all the infinitesimal deformations of $\rho$ constrained by the condition  that the degree $i$ cohomology of the corresponding local system $L_\rho$ has dimension $\ge k$? More precisely, fixing $n$, define the cohomology jump locus $\sV^i_k$ as the set of all such representations.  $\sV^i_k$ has a natural scheme structure when $X$ is a finite CW-complex, and we are asking for a description of the formal scheme $\sV^i_{k,{(L)}}$ at the point $L$.

A nice answer to this question was given in certain cases. Define  the resonance variety $\cR^i_k$ as the set consisting of $\omega\in H^1(X,\cc)$ such that the degree $i$ cohomology of the cup-product complex $(H^\ubul(X,\cc),\omega \cup . )$ has dimension $\ge k$. $\cR^i_k$ also has a natural scheme structure. When $X$ is the complement of a complex hyperplane arrangement and $\rho=\bone$ is the trivial rank $n=1$ representation, Esnault, Schechtman, and Viehweg \cite{ESV} showed that there is an isomorphism of reduced formal germs
\begin{equation}\label{eqHA}
(\sV^i_{k})^{red}_{(\bone)} \cong (\cR^i_{k})^{red}_{(0)}.
\end{equation}
This result has been generalized further by Dimca, Papadima, and Suciu \cite{dps} and recently by Dimca and Papadima \cite{dp}. Their more general result identifies $(\sV^i_k)_{(\bone)}^{red}$ for rank $n\ge 1$ representations on a finite CW-complex $X$ with the reduced formal germ at the origin of a space $\cR^i_k$ defined by replacing the cup-product complexes with Aomoto complexes of the differential graded Lie algebra (DGLA) $\sA^\ubul\otimes_\bC \mathfrak{gl}(\bC^r)$, where $\sA^\ubul$ is any commutative differential graded algebra (CDGA) homotopy equivalent with Sullivan's CDGA $\Omega^\ubul(X,\bC)$ of piecewise smooth $\bC$-forms. In particular, (\ref{eqHA}) is recovered since in that case $\Omega^\ubul(X,\bC)$ is formal, that is, it is homotopy equivalent with its cohomology.

In this article we generalize further these results by providing a description of the formal germ $\sV^i_{k,{(\rho)}}$ at any representation $\rho$ of any rank. In other words, we deal with possibly non-reduced formal germs and with possibly non-trivial local systems.

The deformation problem with cohomology constraints can also be posed for different objects.  For addressing all deformation problems with cohomology constraints at once, we provide a unified framework via differential graded Lie algebras (DGLA), in a sense which we describe next. 

By a deformation problem we mean describing the formal germ $\cM_{(\rho)}$ at some object $\rho$ in some moduli space $\cM$. This is equivalent to describing the corresponding functor on Artinian local algebras which we denote also by $\cM_{(\rho)}$. In fact, this functor is usually well-defined even if the moduli space is not. In practice, every deformation problem over a field of characteristic zero is governed by a DGLA $C$ depending on $(\cM,\rho)$. This means that the functor $\cM_{(\rho)}$ is naturally isomorphic to the deformation functor $\Def(C)$ canonically attached to $C$ via solutions of the Maurer-Cartan equation modulo gauge. An answer to the deformation problem is then obtained by replacing $C$ with a homotopy equivalent DGLA $D$ with enough finiteness conditions which make $\Def(C)=\Def(D)$ representable by an honest space, providing a simpler description of $\cM_{(\rho)}$ than the original definition. A particularly nice answer to the deformation problem is achieved when $D$ can be taken to be the cohomology $H^\ubul(C)$ of $C$, that is, when $C$ is formal. For an overview of this subject, see \cite{M}.

By a deformation problem with cohomology constraints we mean that we have formal germs $\sV^i_{k ,(\rho)}\subset \cM_{(\rho)}$ of objects with a cohomology theory constrained by the condition  that the degree $i$ cohomology has dimension $\ge k$. Then we would like to describe $\sV^i_{k ,(\rho)}$. This is equivalent to describing the corresponding functor on Artinian local algebras which we denote also by $\sV^i_{k ,(\rho)}$, and which in fact is usually well-defined even if the formal germs are not. The point of this article is to stress that, in practice, a deformation problem with cohomology constraints over a field of characteristic zero is governed by a pair $(C,M)$ of a DGLA together with a module over it. Given any such pair, we will canonically define {\it cohomology jump functors} $\Def^i_k(C,M)$. When $\sV^i_{k ,(\rho)}\cong\Def^i_k(C,M)$ as subfunctors of $\cM_{(\rho)}\cong\Def(C)$, we obtain an answer to the deformation problem with cohomology constraints by replacing $(C,M)$ with a homotopy equivalent pair $(D,N)$ with enough finiteness conditions which makes $\Def^i_k(C,M)\cong\Def^i_k(D,N)$ representable by an honest space (like $\cR^i_k$ above). A particularly nice answer is achieved when $(D,N)=(H^\ubul(C),H^\ubul(M))$, that is when the pair $(C,M)$ is formal.

In this paper we consider the deformation problem with cohomology constraints for linear representations of fundamental groups, local systems, holomorphic vector bundles, and Higgs bundles.

The idea of using DGLA pairs is already implicit in \cite{M-a}, where a functor $\Def_\chi$ of semi-trivialized deformations is attached to a DGLA map $\chi:C\ra C'$. Such situation arises for a DGLA pair $(C,M)$ by setting $C'=\End^\ubul(M)$ with the induced natural DGLA structure. It is shown in this case in \cite{M-a} that the image of $\Def_\chi\ra \Def(C)$ describes the deformation problem with no-change-in-cohomology constraint. Therefore our $\Def^i_k(C,M)$ can be seen as refinements of $\Def_\chi$ for $\chi:C\ra \End^\ubul(M)$.

A theorem due to  Lurie \cite{Lu} and Pridham \cite{Pr} in the framework of derived algebraic geometry states that, with the appropriate axiomatization, every infinitesimal deformation problem is governed by a DGLA and that a converse holds. A natural question is whether this can be extended to an equivalence between infinitesimal deformation problems with cohomology constraints and DGLA pairs.

This article puts together, simplifies, and extends previous two articles by the second author \cite{w,wb}. The DGLA pairs were introduced in the first version of this article \cite{wb} by the second author to address the reduced structure of the cohomology jump loci. Since \cite{w, wb} will not be published, we will provide complete arguments, even though they may have already appeared in \cite{w, wb}. The main concept introduced in this second version is that of cohomology jump functors of a DGLA pair. Since this refines the deformation functor, and leads to more direct proofs of even stronger results, we feel that this approach is the closest to a hypothetic final answer to the general problem of infinitesimal deformations with cohomology constraints.
 
Let us describe next in more detail the results of this article.

\subsection{Cohomology jump loci of complexes.}\label{subsCpx} Let $M$ be a finitely generated module over a Noetherian ring $R$. Let 
$
G\;\displaystyle{\mathop{ \ra }^d}\; F \ra M\ra 0
$
be a presentation of $M$ by finitely generated free $R$-modules. Then  the ideal $$J_k(M)=I_{rank(F)-k}(d)$$ of minors of size $rank(F)-k$ of a matrix representing $d$ does not depend on the choice of presentation. This is result goes back to J.W. Alexander and H. Fitting. We generalize it to complexes as follows.

Now let $E^\ubul$ be a complex of $R$-modules, bounded above, such that $H^i(E^\ubul)$ is a finitely generated $R$-module for every $i$. By a lemma of Mumford \cite[III.12.3]{h}, there exists a complex $F^\ubul$ of finitely generated free $R$-modules, and a morphism of complexes $g: F^\ubul\to E^\ubul$ which is a quasi-isomorphism. We define the {\it cohomology jump ideals of $E^\ubul$}  to be 
$$
J^i_k(E^\ubul)=I_{rank(F^i)-k+1}(d^{i-1}\oplus d^{i})
$$
where $d^{i-1}: F^{i-1}\to F^i$ and $d^i: F^i\to F^{i+1}$ are the differentials of the complex $F^\ubul$. We show that $J^i_k(E^\ubul)$ does not depend on the choice of $F^\ubul$ (Definition-Proposition \ref{defprop}). This is the content of Section \ref{secCJI}.

We are interested in the subscheme of $\spec(R)$ associated to such $J^i_k(E^\ubul)$. This setup occurs in many situations, for example when $E^\ubul$ can obtained from topology or from DGLA pairs as below.



\subsection{Cohomology jump loci of DGLA pairs.} Let $C=(C^\ubul, d_C)$ be a DGLA over $\cc$. The deformation functor $\Def(C)$ attached to $C$ is a functor from the category of Artinian local algebras to the category of sets.  Let $M=(M^\ubul,d_M)$ be a differential graded  module over $C$. Then, using \ref{subsCpx}, we define and study in Section \ref{secDGLA} the {\it cohomology jump functors} $\Def^i_k(C,M)$ as subfunctors of $\Def(C)$. Let us state the crucial property next.

In general the DGLA that governs a deformation problem has infinite dimension on each degree. The following result of Deligne-Goldman-Millson-Schlessinger-Stasheff allows one to replace the DGLA with a finite dimensional one within the same homotopy equivalence class, when such a DGLA is available. 

\begin{theorem}[\cite{gm}]\label{gm0}
The cohomology functor $\Def(C)$ only depends on the 1-homotopy type of $C$. More precisely, if a morphism of DGLA $f: C\to D$ is 1-equivalent, then the induced transformation on functors $f_*: \Def(C)\to \Def(D)$ is an isomorphism. 
\end{theorem}
The familiar notions of $i$-equivalence and $i$-homotopy extend easily to DGLA pairs, see Section \ref{secDGLA}. We also extend the last theorem to DGLA pairs:
\begin{theorem}\label{independence}
The cohomology jump functor $\Def^i_k(C, M)$ only depends on the $i$-homotopy type of $(C, M)$. More precisely, if a morphism of DGLA pairs $g=(g_1, g_2): (C, M)\to (D, N)$ satisfies that $g_1$ is 1-equivalent and $g_2$ is $i$-equivalent, then the induced transformation on functors $g_*: \Def^i_k(C, M)\to \Def^i_k(D, N)$ is an isomorphism. 
\end{theorem}

A typical application of Theorem \ref{gm0} is when $C$ is formal. Similarly, we  define formality for DGLA pairs and make use of it via Theorem \ref{independence} when we consider concrete deformation problems as below.

In Sections \ref{secQR} and \ref{augment} we address the quadratic cones, resonance varieties, and augmentations of DGLA pairs, needed in our analysis of concrete deformation problems.

\subsection{Holomorphic vector bundles.} In Section \ref{holomorphic}, we consider the moduli space $\cM$ of stable rank $n$ holomorphic vector bundles $E$ with vanishing Chern classes on a compact K\"ahler manifold $X$. These holomorphic vector bundles are the ones that admit flat unitary connections. In $\cM$, we consider the cohomology jump loci
$$
\sV^{pq}_k(F)=\{E\in\sM \mid \dim H^q(X, E\otimes_{\sO_X} F \otimes_{\sO_X} \Omega^p_X)\geq k \}
$$
with the natural scheme structure, for fixed $p$ and fixed poly-stable bundle $F$ with vanishing Chern classes. We show that this deformation problem with cohomology contraints is governed by the DGLA pair $(\Ad(\enmo(E)), \Ap(E\otimes F))$ constructed from Dolbeault complexes (Theorem \ref{vb1}). Let 
\begin{align*}
\sQ(E) &=\{\eta\in H^1(X, \enmo(E))\mid
\eta\wedge\eta=0\in H^2(X, \enmo(E))\}, \\
\sR^{pq}_k(E; F) & =\{\eta\in \sQ(E) \mid  \dim H^q(H^\ubul(X, E\otimes F\otimes \Omega^p_X),\eta \wedge\cdot )\geq k\},
\end{align*}
with natural scheme structures defined using \ref{subsCpx}. Formality of the DGLA pair implies:

\begin{thrm}\label{thrmHolVB} Let $X$ be a compact K\"ahler manifold. Let $E$ and $F$ be a stable and, respectively, a poly-stable holomorphic vector bundle with vanishing Chern classes on $X$. Then there is an isomorphism of formal schemes
$$
 \sV^{pq}_k(F)_{(E)}\cong\sR^{pq}_k(E;F)_{(0)}.
$$ 
\end{thrm}

This generalizes the result of  Nadel \cite{n} and Goldman-Millson \cite{gm} that $\sM_{(E)}\cong \sQ(E)_{(0)}$, and it  also generalizes a result of Green-Lazarsfeld \cite{gl1,gl2} for rank $n=1$ bundles. It also implies that if $k=\dim H^q(X, E\otimes F\otimes \Omega^p_X)$, then $\sV^{pq}_k(F)$ has  quadratic algebraic singularities at $E$ (Corollary \ref{corQuad}), a result also shown for $F\otimes\Omega_X^p=\cO_X$ by Martinengo \cite{ma} and the second author \cite{w}.


\subsection{Irreducible local systems and Higgs bundles.} In Section \ref{localsystem}, we consider the moduli space $\mb$ of irreducible rank $n$ local systems $L$ on a compact K\"ahler manifold $X$, and we consider the cohomology jump loci
$$
\cV^i_k(W)=\{ L\in \mb\mid \dim_\bC H^i(X,L\otimes_\bC W)\ge k \}
$$
with the natural scheme structure, for a fixed semi-simple local system $W$ of any rank. The DGLA pair governing this deformation problem with cohomology constraints is $(\Ar(\enmo(L)), \Ar(L\otimes W))$, constructed from the de Rham complex. Parallel results and proofs similar to the case of holomorphic vector bundles hold. Let
$$
\sQ(L) =\{\eta\in H^1(X, \enmo(L)) \mid 
\eta\wedge\eta=0\in H^2(X, \enmo(L))\},$$
$$
\sR^i_k(L;W)  =\{ \eta\in \sQ(L)\mid \dim H^i(H^\ubul(X, L\otimes W), \eta\wedge \cdot)\geq k \},
$$
with the natural scheme structures. 

\begin{thrm}\label{thmIrrLS} Let $X$ be a compact K\"ahler manifold. Let $L$ be an irreducible local system on $X$, and let $W$ be a semi-simple local system. The isomorphism of formal schemes
$$
(\mb)_{(L)}\cong \sQ(L)_{(0)}
$$
induces an isomorphism
$$
\sV^i_k(W)_{(L)}\cong\, \sR^i_k(L;W)_{(0)}. 
$$
\end{thrm}
The proof of this result generalizes the main ``strong linearity'' result of  Popa-Schnell as stated in \cite[Theorem 3.7]{PS}, proved there for rank one local systems $E$, $W=\bC_X$, and $X$ a smooth projective complex variety, see Remark \ref{rmkPS}.

Higgs bundles are similarly treated in Section \ref{secHB}, via a DGLA pair arising from the Higgs complex.

\subsection{Representations of the fundamental group.}
Also in Section \ref{localsystem}, we look at representations of the fundamental group. This case is closely related to the case of local systems. This relation, at the level of deformations, is a particular case of the relation between the deformation functors of an augmented DGLA pair and those of the DGLA pair itself, see Theorem \ref{mainaug}.

Let $X$ be a smooth manifold which is of the homotopy type of a finite type CW-complex, and let $x\in X$ be a base point. The set of group homomorphisms $\homo(\pi_1(X, x), GL(n, \mathbb{C}))$ has naturally a scheme structure. We denote this scheme by $\mathbf{R}(X, n)$. Every closed point $\rho\in\mathbf{R}(X, n)$ corresponds to a rank $n$ local system $L_\rho$ on $X$. Let $W$ be a local system of any rank on $X$. In $\mathbf{R}(X, n)$, we define the cohomology jump loci 
$$\ti\sV^i_k(W)=\{\rho\in \bR(X, n)\,|\, \dim H^i(X, L_\rho\otimes_\bC W)\geq k\}
$$
with the natural scheme structure (these were denoted $\sV^i_k$ in \ref{subsMot}).

 We show that an augmented DGLA pair $(\Ar(\enmo(L_\rho)), \Ar(L_\rho\otimes_\bC W); \varepsilon)$ governs this deformation problem with cohomology constraints (Theorem \ref{thmRP1}). This generalizes the result of Goldman-Millson \cite{gm} who showed that the deformation problem without cohomology constraints is governed by the augmented DGLA $(\Ar(\enmo(L_\rho)); \varepsilon)$. This also generalizes the result of Dimca-Papadima \cite{dp} mentioned in \ref{subsMot}. In \cite{dp}, $X$ is allowed to be a connected CW-complex of finite type by replacing the de Rham complex with Sullivan's  de Rham complex, but, for simplicity, we opted to leave out this topological refinement. 

Thus, the formal scheme of $\ti\sV^i_k(\bC_X^n)$ at the trivial representation only depends on the $k$-homotopy type of the topological space $X$, generalizing a result of \cite{dp} for the underlying reduced germs.

Let
$$
\sQ(\rho)=\{\eta\in Z^1(\pi_1(X), \mathfrak{gl}(n, \cc)_{\ad \rho})\,|\,\bar\eta\wedge\bar\eta=0 \in H^2(X, \enmo(L_\rho))\},$$
$$
\sR^i_k(\rho,W)=\{\eta\in  \sQ(\rho)\,|\, \dim H^i(H^\ubul(X, L_\rho\otimes_\bC W), \bar\eta\wedge\cdot)\geq k\},
$$
with the natural scheme structures, where $Z^1$ stands for the vector space of 1-cocycles, and $\bar\eta$ is the image of $\eta$ in cohomology.

\begin{thrm}\label{thmRPP}
Let $X$ be a compact K\"ahler manifold,  $\rho\in\mathbf{R}(X,n)$ be a semi-simple representation, and $W$ a semi-simple local system on $X$. Then
$$
\ti\sV^i_k(W)_{(\rho)}\cong \sR^i_k(\rho,W)_{(0)}.
$$
\end{thrm}

This generalizes the result of  Simpson \cite{s1} that $
\mathbf{R}(X, n)_{(\rho)}\cong \sQ(\rho)_{(0)}
$. With the same assumptions, if in addition $k=\dim H^i(X,L_\rho)$, then $\ti\sV^i_k(W)$ has quadratic singularities at $\rho$ (Corollary \ref{corQRep}).

\subsection{Other consequences of formality.}  Theorems \ref{thrmHolVB}, \ref{thmIrrLS}, \ref{thmRPP} describing the local structure of cohomology jump loci $\sV^i_k$ in terms of cohomology resonance loci $\sR^i_k$ are consequences of the formality of the DGLA pair governing the corresponding deformation problem with cohomology constraints. In Section \ref{secIneq}, we show that formality for a DGLA pair $(C,M)$ leads to more information about the geometry of cohomology resonance loci and about the possible shapes of the sequence of Betti numbers $\dim H^i(M)$. This puts together and extends to DGLA pairs a method which was previously employed in different setups by Lazarsfeld-Popa \cite{LP}, the first author \cite{B-h}, and Popa-Schnell \cite{PS}.

\subsection{Analytic and \'etale local germs.}
According to Artin's approximation theorem \cite{ar}, two analytic germs $(X, x)$ and $(Y, y)$  are isomorphic if and only if the formal schemes $X_{(x)}$ and $Y_{(y)}$ are isomorphic. Furthermore, Artin also showed in \cite{ar1} that as \'etale local germs $(X, x)$ and $(Y, y)$ are isomorphic in the algebraic category. Thus, our results on isomorphisms between formal schemes can be stated as isomorphisms between analytic germs and also between algebraic \'etale germs.

\subsection{Notation.}
Throughout this paper, all rings are defined over $\bC$. By an Artinian local algebra, we mean an Artinian local algebra which is of finite type over $\bC$. Denote the category of Artinian local algebras with local homomorphisms by $\Art$ and the category of sets by $\Set$. Suppose $\mathbf{F}$ is a functor from $\Art$ to $\Set$. We shall say a formal scheme $\mathcal{X}$ consisting only one closed point (or a complete local ring $R$ resp.) prorepresents the functor $\mathbf{F}$, if $\homo(\Gamma(\mathcal{X}, \sO_\mathcal{X}), -)$ (resp. $\homo(R, -)$) is naturally isomorphic to the functor $\mathbf{F}$. By abusing notation, we will frequently use the same letter to denote a closed point in some moduli space and the object the closed point represents. Also by abusing notation, we will frequently use a formal scheme $\mathcal{X}$ (supported at a point) to denote the functor it prorepresents, i.e., $\homo(-, \mathcal{X}): \Art\to \Set$. 

\begin{acknowledgements}
We thank Donu Arapura, Christian Schnell, Mihnea Popa for several helpful conversations and to the referee for important comments. 
\end{acknowledgements}














\section{Cohomology jump loci of complexes}\label{secCJI}
Let $R$ be a noetherian ring, and let $E^\ubul$ be a complex of $R$-modules, bounded above. Suppose $H^i(E^\ubul)$ is a finitely generated $R$-module. In this section, we define the notion of cohomology jump ideals for the complex $E^\ubul$. Throughout this section, we assume all complexes of $R$-modules are bounded above and have finitely generated cohomology. 

First, we want to replace $E^\ubul$ by a complex of finitely generated free $R$-modules. This is achieved by a lemma of Mumford. 
\begin{lemma}[\cite{h}-III.12.3]\label{lemM}
Let $R$ and $E^\ubul$ be defined as above. There exists a bounded above complex $F^\ubul$ of finitely generated free $R$-modules and a morphism of complexes $\phi: F^\ubul\to E^\ubul$ which is a quasi-isomorphism. 
\end{lemma}

\begin{defprop}\label{defprop}{\it 
Under the above notations, we define the {\bf cohomology jump ideals} to be 
$$
J^i_k(E^\ubul)=I_{rank(F^i)-k+1}(d^{i-1}\oplus d^{i})
$$
where $I$ denotes the determinantal ideal, $d^{i-1}: F^{i-1}\to F^i$ and $d^i: F^i\to F^{i+1}$ are the differentials of the complex $F^\ubul$. Then $J^i_k(E^\ubul)$ does not depend on the choice of $F^\ubul$. }
\end{defprop}
\begin{proof} This is a generalization of the proof of the Fitting Lemma from \cite[20.4]{eisenbud}. We can assume $R$ is local. By \cite[Proposition 4.4.2]{R}, $F^\ubul$ has a unique minimal free resolution $G^\ubul \rightarrow F^\ubul$. Let $\rho:G^\ubul\rightarrow E^\ubul$ the composition with $\phi$. If $\bar{\phi}: \bar{F}^\ubul\rightarrow E^\ubul$ is another free resolution of $E^\ubul$ with finite rank terms, then by \cite[Theorem 3.1.7]{R}, there exists a map $\beta: G^\ubul\rightarrow \bar{F}^\ubul$ unique up to homotopy such that $\bar{\phi}\beta$ is homotopic with $\rho$. Hence $\beta$ is a quasi-isomorphism, and so $G^\ubul$ is a minimal free resolution of $\bar{F}^\ubul$ also. Thus, it is enough to prove  that $J^i_k(E^\ubul)$ is the same if computed with $F^\ubul$ and $G^\ubul$. 

By \cite[Proposition 4.4.2]{R}, $F^\ubul$ is a direct sum of $G^\ubul$ with a direct sum of shifts of the trivial complex 
$$ 0\rightarrow R{\mathop{\rightarrow}^1} R\rightarrow 0.$$ It is enough, by induction, to assume that only one such shifted trivial complex is added to $G^\ubul$ to obtain $F^\ubul$. Fix $i$. There are four shifts of the trivial complex that can be added to $G^{i-1}\rightarrow G^i\rightarrow G^{i+1}$. Let $r$ be the rank of $G^i$ and $M$ the matrix of $d_G^{i-1}\oplus d_G^{i}$. The ideals $I_{rank(F^i)-k}(d_F^{i-1}\oplus d_F^{i})$ for each of the four possible cases  are: 
\begin{align*}
I_{r-k}\begin{pmatrix} M & 0  \end{pmatrix},  I_{r+1-k}\begin{pmatrix} M & 0 & 0 \\ 0 & 1 & 0 \end{pmatrix}, I_{r+1-k}\begin{pmatrix} M & 0 \\ 0 & 0 \\ 0 & 1 \end{pmatrix}, I_{r-k}\begin{pmatrix} M  \\ 0 \end{pmatrix},
\end{align*}
and all are equal to $I_{r-k}(M)$ as we wanted to show.
\end{proof}

\begin{cor}\label{idealindep}
If $E^\ubul$ is quasi-isomorphic to $E'^\ubul$, then $J^i_k(E^\ubul)=J^i_k(E'^\ubul)$. 
\end{cor}

\begin{cor}\label{tensor}
Let $R$ and $E^\ubul$ be defined as above, and let $S$ be a noetherian $R$-algebra. Moreover, suppose $E^\ubul$ is a complex of flat $R$-modules, then $J^i_k(E^\ubul)\otimes_R S=J^i_k(E^\ubul\otimes_R S)$, where we regard $E^\ubul\otimes_R S$ as a complex of $S$ modules. 
\end{cor}
\begin{proof}
By Lemma \ref{lemM}, there is a quasi-isomorphism $F^\ubul\to E^\ubul$, where $F^\ubul$ is a bounded above complex of finitely generated free $R$-modules. Since $E^\ubul$ is bounded above and flat, $F^\ubul\otimes_R S$ is quasi-isomorphic to $E^\ubul \otimes_R S$. Thus, $J^i_k(E^\ubul\otimes_R S)$ can be computed as determinantal ideals of $F^\ubul\otimes_R S$. Hence, the corollary follows from the fact that taking determinantal ideals commutes with taking tensor product. 
\end{proof}

When $R$ is a field, by definition $J^i_k(E^\ubul)=0$ if $\dim H^i(E^\ubul)\geq k$ and $J^i_k(E^\ubul)=R$ if $\dim H^i(E^\ubul)<k$. Thus, we have the following.

\begin{cor}\label{corFl}
Suppose $E^\ubul$ is a complex of flat $R$-modules. Then for any maximal ideal $m$ of $R$, $J^i_k(E^\ubul)\subset m$ if and only if $\dim_{R/m} H^i(E^\ubul\otimes_R R/m)\geq k$. 
\end{cor}

Next, we address a partial generalization of Corollary \ref{idealindep}. 

\begin{defn} A morphism of complexes is $q$-{\bf equivalent} if it induces an isomorphism on cohomology up to degree $q$ and a monomorphism at degree $q+1$. For example, $\infty$-equivalent means quasi-isomorphic. 
\end{defn}

\begin{prop}\label{propQequiv}
Let $(R,m)$ be a noetherian local ring and let $f: E^\ubul\ra E'^\ubul$ be a $q$-equivalence between two bounded above complexes of free $R$-modules with finitely generated cohomology. If $f\otimes \id_{R/m}: E^\ubul\otimes_R R/m\ra E'^\ubul\otimes_R R/m$ is also a $q$-equivalence, then $J^i_k(E^\ubul)=J^i_k(E'^\ubul)$ for $i\le q$.
\end{prop}
\begin{proof}
Let $\phi: F^\ubul\to E^\ubul$ and $\phi': F'^\ubul\to E'^\ubul$ be the minimal free resolutions of $E^\ubul$ and $E'^\ubul$ respectively. Since $F^\ubul$ is a complex of free $R$-modules, we can lift the composition $f\circ \phi: F^\ubul\to E'^\ubul$ via $\phi'$ to $g:F\to F'$. Thus, we obtain the following diagram,
\begin{equation*}
\xymatrix{
F^\ubul\ar[r]^{g}\ar[d]^{\phi}&F'^\ubul\ar[d]^{\phi'}\\
E^\ubul\ar[r]^{f}&E'^\ubul
}
\end{equation*}
where $\phi$ and $\phi'$ are $\infty$-equivalent, and $f$ is $q$-equivalent. Since the diagram commutes, $g$ is also $q$-equivalent. Taking the tensor product of the above diagram with $R/m$ over $R$ gives us another diagram,
\begin{equation*}
\xymatrix{
F^\ubul\otimes_{R}R/m\ar[r]^{\bar{g}}\ar[d]^{\bar\phi}&F'^\ubul\otimes_{R}R/m\ar[d]^{\bar\phi'}\\
E^\ubul\otimes_{R}R/m\ar[r]^{\bar{f}}&E'^\ubul\otimes_{R}R/m
}
\end{equation*}
where $\bar{f}$ is $q$-equivalent by assumption. Since $E^\ubul$, $E'^\ubul$, $F^\ubul$ and $F'^\ubul$ are complexes of free, hence flat, $R$-modules, $\bar\phi$ and $\bar\phi'$ are $\infty$-equivalent. Therefore, $\bar{g}$ is also $q$-equivalent. 

Since $F$ and $F'$ are minimal, the differentials in $F^\ubul\otimes_{R}R/m$ and $F'^\ubul\otimes_{R}R/m$ are all zero. Therefore, $\bar{g}: F^\ubul\otimes_{R}R/m\to F'^\ubul\otimes_{R}R/m$ being $q$-equivalent means 
$$\bar{g}^i: F^i\otimes_{R}R/m\to F'^i\otimes_{R}R/m$$
is an isomorphism for $i\leq q$ and a monomorphism for $i=q+1$. In particular, $rank(F^i)=rank(F'^i)$ for $i\leq q$. 

By definition, $J^i_k(E^\ubul)=J^i_k(F^\ubul)$ and $J^i_k(E'^\ubul)=J^i_k(F'^\ubul)$. Hence we only need to show $J^i_k(F^\ubul)=J^i_k(F'^\ubul)$ for $i\le q$. Recall that $J^i_k(F^\ubul)=I_{rank(F^i)-k+1}(d^{i-1}\oplus d^i)$, where $d^{i-1}$ and $d^i$ are the differentials in $F^\ubul$. Notice that 
$$I_{rank(F^i)-k+1}(d^{i-1}\oplus d^i)=\sum_{0\leq j\leq rank(F^i)-k+1}I_{j}(d^{i-1})\cdot I_{rank(F^i)-k+1-j}(d^{i}).$$
Since $rank(F^i)=rank(F'^i)$ for $i\leq q$, to show $J^i_k(F^\ubul)=J^i_k(F'^\ubul)$ for $i\le q$, it suffixes to show $I_j(d^i)=I_j(d'^i)$ for any $j\in \nn$ and $i\leq q$, where $d'^i$ is the differential in $F'^\ubul$. This follows from following two statements, which will be proved in the next two lemmas.
\begin{enumerate}
\item $g^i: F^i\to F'^i$ is an isomorphism for $i\leq q$;
\item $g^{q+1}: F^{q+1}\to F'^{q+1}$ is injective and its image is a direct summand of $F'^{q+1}$. 
\end{enumerate}
\end{proof}
\begin{lemma}
Let $(R, m)$ be a noetherian local ring. Let $h: M\to M'$ be a morphism between finite free $R$-modules. Suppose $h\otimes \id_{R/m}: M\otimes_R R/m\to M'\otimes_R R/m$ is an isomorphism. Then $h$ is an isomorphism. 
\end{lemma}
\begin{proof}
The composition $M\stackrel{h}{\to}M'\to M'\otimes_R R/m$ is surjective. By Nakayama's lemma, $h$ is surjective. Since $M'$ is free, we have a short exact sequence 
$$0\to Ker(h)\otimes_R R/m\to M\otimes_R R/m\to M'\otimes_R R/m\to 0.$$
Since $M\otimes_R R/m\to M'\otimes_R R/m$ is an isomorphism, $Ker(h)\otimes_R R/m=0$. Hence, $Ker(h)=0$ by Nakayama's lemma. 
\end{proof}
\begin{lemma}
Let $(R, m)$ be a noetherian local ring. Let $h: M\to M'$ be a morphism between finite free $R$-modules. Suppose $h\otimes \id_{R/m}: M\otimes_R R/m\to M'\otimes_R R/m$ is injective. Then $h$ is injective, and the cokernel of $h$ is a free $R$-module. 
\end{lemma}
\begin{proof}
Denote the kernel, image and cokernel of $h$ by $Ker(h)$, $Im(h)$ and $Coker(h)$ respectively. Then we have two short exact sequences,
\begin{equation}\label{ses1}
0\to Ker(h)\to M\to Im(h)\to 0
\end{equation}
and
\begin{equation}\label{ses2}
0\to Im(h)\to M'\to Coker(h)\to 0
\end{equation}
Since $M$ and $M'$ are free $R$-modules, we have the following exact sequences by taking tensor with $R/m$. 
$$0\to Tor_1(Im(h), R/m)\to Ker(h)\otimes R/m\to M\otimes R/m\to Im(h)\otimes R/m\to 0$$
and
$$0\to Tor_1(Coker(h), R/m)\to Im(h)\otimes R/m\to M'\otimes R/m\to Coker(h)\otimes R/m\to 0$$
where all the tensor and $Tor$ are over $R$. 

Notice that the morphism $h\otimes \id_{R/m}: M\otimes_R R/m\to M'\otimes_R R/m$ factors as $M\otimes_R R/m\to Im(h)\otimes R/m\to M'\otimes R/m$. Since $h\otimes_R \id_{R/m}$ is injective and since $M\otimes_R R/m\to Im(h)\otimes R/m$ is obviously surjective, $Im(h)\otimes R/m\to M'\otimes R/m$ must be injective. Therefore, we have the vanishing $Tor_1(Coker(h), R/m)=0$. Over a noetherian local ring, this means $Coker(h)$ is free. By short exact sequences (\ref{ses2}) and (\ref{ses1}), we can conclude $Im(h)$ and $Ker(h)$ are both free. Thus, $M$ splits as a direct sum of free $R$-modules $Ker(h)$ and $Im(h)$. Now, since $h\otimes \id_{R/m}: M\otimes_R R/m\to M'\otimes_R R/m$ is injective, clearly $Ker(h)=0$. 
\end{proof}

\begin{rmk}
The assertion of Proposition \ref{propQequiv} is not necessarily true only with the assumption that $E^\ubul$ and $E'^\ubul$ are $q$-equivalent. This can be seen by taking a zero-complex and a free resolution of a non-free $R$-module.
\end{rmk}

\section{Cohomology jump loci of DGLA pairs}\label{secDGLA}
In this section we recall the definition of the deformation functor of a DGLA, we define DGLA pairs and their cohomology jump functors, and we prove Theorem \ref{independence} on the invariance of the cohomology jump functors under a change of the DGLA pair.

Firstly, recall the definition of a DGLA over $\mathbb{C}$, for example from \cite{gm}:

\begin{defn}
A DGLA consists of the following set of data,
\begin{enumerate}
\item a graded vector space $C=\bigoplus_{i\in \mathbb{N}}C^i$ over $\mathbb{C}$,
\item a Lie bracket which is bilinear, graded skew-commutative, and satisfies the graded Jacobi identity, i.e., for any $\alpha\in C^i, \beta\in C^j$ and $\gamma\in C^k$,
$$[\alpha, \beta]+(-1)^{ij}[\beta, \alpha]=0$$
and
$$(-1)^{ki}[\alpha, [\beta, \gamma]]+(-1)^{ij}[\beta, [\gamma, \alpha]]+(-1)^{jk}[\gamma, [\alpha, \beta]]=0$$
\item a family of linear maps, called the differential maps, $d^i: C^i\to C^{i+1}$, satisfying $d^{i+1}d^i=0$ and the Leibniz rule, i.e., for $\alpha\in C^i$ and $\beta\in C$
$$d[\alpha, \beta]=[d\alpha, \beta]+(-1)^i[\alpha, d\beta]$$
where $d=\sum d^i: C\to C$. 
\end{enumerate}
A homomorphism of DGLAs is a linear map which preserves the grading, Lie bracket, and the differential maps. 
\end{defn}
We denote this DGLA by $(C, d)$, or $C$ when there is no risk of confusion. 

\begin{defn}\label{module}
Given a DGLA $(C, d_C)$, we define a \textbf{module} over $(C, d_C)$ to be the following set of data,
\begin{enumerate}
\item a graded vector space $M=\bigoplus_{i\in \mathbb{N}} M^i$ together with a bilinear multiplication map $C\times M\to M$, $(a, \xi)\mapsto a\xi$, such that for any $\alpha\in C^i$ and $\xi\in M^j$, $\alpha\xi \in M^{i+j}$. And furthermore, for any $\alpha\in C^i, \beta\in C^j$ and $\zeta\in M$, we require
$$[\alpha, \beta]\zeta=\alpha(\beta\zeta)-(-1)^{ij}\beta(\alpha\zeta).$$
\item a family of linear maps $d^i_M: M^i\to M^{i+1}$ (write $d_M=\sum_{i\in\zz} d^i_M: M\to M$), satisfying $d^{i+1}_M d^i_M=0$. And we require it to be compatible with the differential on $C$, i.e., for any $\alpha\in C^i$, 
$$d_M(\alpha\xi)=(d_C\alpha)\xi+(-1)^i\alpha(d_M\xi).$$
\end{enumerate}
\end{defn}

We will call such a module by a $(C, d_C)$-module or simply a $C$-module. 

\begin{defn}
A \textbf{homomorphism} of $(C, d_C)$-modules $f: (M, d_M)\to (N, d_N)$ is a linear map $f: M\to N$ which satisfies
\begin{enumerate}
\item $f$ preserves the grading, i.e., $f(M^i)\subset N^i$,
\item $f$ is compatible with multiplication by elements in $C$, i.e., $f(\alpha\xi)=\alpha f(\xi)$, for any $\alpha\in C$ and $\xi \in M$,
\item $f$ is compatible with the differentials, i.e., $f(d_M\alpha)=d_N f(\alpha)$.
\end{enumerate}
\end{defn}

Fixing a DGLA $(C, d_C)$, the category of $C$-modules is an abelian category. 

\begin{defn}\label{defnHot}
A \textbf{DGLA pair} is a DGLA $(C, d_C)$ together with a $(C, d_C)$-module $(M, d_M)$. Usually, we write such a pair simply by $(C, M)$.  A homomorphism of DGLA pairs $g: (C, M)\to (D, N)$ consists of a map $g_1: C\to D$ of DGLA and a $C$-module homomorphism $g_2: M\to N$, considering $N$ as a $C$-module induced by $g_1$. For $q\in \nn\cup \{\infty\}$, we call $g$ a $q$-\textbf{equivalence} if $g_1$ is 1-equivalent and $g_2$ is $q$-equivalent. Moreover, we define two DGLA pairs to be of the same $q$-\textbf{homotopy type}, if they can be connected by a zig-zag of $q$-equivalences. Two DGLA pairs have the same {\textbf{homotopy type}} if they have the same $\infty$-homotopy type.
\end{defn}

\begin{defn} Let $(C, M)$ be a DGLA pair. Then $(H^\ubul(C),0)$, the cohomology of $C$ with zero differentials, is a DGLA, and $(H^\ubul(M),0)$, the cohomology of $M$ with zero differentials, is an $H^\ubul(C)$-module. We call the DGLA pair $(H^\ubul(C), H^\ubul(M))$ the \textbf{cohomology DGLA pair} of $(C, M)$. 
\end{defn}

\begin{assumption} From now, for a DGLA pair $(C,M)$ we always assume that $M$ is bounded above as a complex and $H^j(M)$ is a finite dimensional $\cc$-vector space for every $j\in \zz$.
\end{assumption}

\begin{defn}
We say the DGLA pair $(C, M)$ is $q$-\textbf{formal} if $(C, M)$ is of the same $q$-homotopy type as $(H^\ubul(C), H^\ubul(M))$. A pair is {\bf formal} is it is $\infty$-formal.
\end{defn}

Given a DGLA pair, we can abstractly define the space of flat connections and the cohomology jump loci as functors from $\Art$ to $\Set$. We will be mainly interested in the case when these functors are prorepresentable. 

Given a DGLA $(C, d)$ over $\mathbb{C}$ together with an Artinian local algebra $A$, a groupoid $\mathcal{C}(C, A)$ is defined in \cite{gm}. We recall this definition. $C\otimes_{\cc}A$ is naturally a DGLA by letting $[\alpha\otimes a, \beta\otimes b]=[\alpha, \beta]\otimes ab$ and $d(\alpha\otimes a)=d\alpha \otimes a$. Let $m$ be the maximal ideal in $A$. Then under the same formula, $C\otimes_{\cc} m$ is also a DGLA. Since $(C\otimes_{\cc}m)^0=C^0\otimes_{\cc}m$ is a nilpotent Lie algebra, the Campbell-Hausdorff multiplication defines a nilpotent Lie group structure on the space $C^0\otimes m$. We denote this Lie group by $\exp(C^0\otimes m)$. Now, a element $\lambda\in C^0\otimes m$ acts on $C^1\otimes m$ by
$$\overline\exp(\lambda): \alpha \mapsto \exp(\ad \lambda)\alpha+\frac{1-\exp(\ad \lambda)}{\ad \lambda}(d\lambda)$$
in terms of power series. This is a group action for the group $\exp(C^0\otimes m)$ on $C^1\otimes m$.

\begin{defn}\label{cat}
Category $\sC(C; A)$ is defined to be the category with objects
$$\obj \sC(C; A)=\{\omega\in C^1\otimes_{\cc}m\;|\; d\omega+\frac{1}{2}[\omega, \omega]=0\},$$
and with the morphisms between two elements $\omega_1$, $\omega_2$
$$\morph(\omega_1, \omega_2)=\{\lambda\in C^0\otimes m\;|\; \overline\exp(\lambda)\omega_1=\omega_2\}.$$
Define the \textbf{deformation functor} to be the functor 
$$
\Def(C): A\mapsto \iso\sC(C; A)
$$
from $\Art$ to $\Set$. Here we denote the set of isomorphism classes of a category by $\iso$. 
\end{defn}

\begin{defn}
Given any $\omega\in \obj\sC(C; A)$ and a $C$-module $M$, we can associate an {\bf Aomoto complex} to it: 
\begin{equation}\label{eqAo}
(M\otimes_\bC A, d_\omega)
\end{equation}
 with  $$d_\omega:=d\otimes \id_A+\omega.$$ The condition $d\omega+\frac{1}{2}[\omega, \omega]=0$ implies $d_\omega\circ d_\omega=0$. 
\end{defn}

\begin{lemma}\label{lemFGCoh} $(M\otimes_\bC A,d_\omega)$ has finitely generated cohomology over $A$.
\end{lemma} 
\begin{proof}
The finite decreasing filtration $M\otimes_\bC m^s$ of $M\otimes_\bC A$ is compatible with $d_\omega$. Consider the associated spectral sequence
$$
E_1^{s,t}=H^{s+t}(M\otimes_\bC m^s/m^{s+1},d_\omega)\Rightarrow H^{s+t}(M\otimes_\bC A,d_\omega),
$$
which degenerates after finitely many pages. It is enough to show that $E_1^{s,t}$ are finitely generated. However, this follows from the fact that $d_\omega=d\otimes id_A$ on $M\otimes m^s/m^{s+1}$, together with our assumption that $(M,d)$ has finitely generated cohomology.
\end{proof}
 
\begin{prop}\label{propexp}
Given any $\lambda\in C^0\otimes m$, the morphism $\overline\exp(\lambda): \omega_1\to \omega_2$ in $\sC(C; A)$ induces functorially a morphism between complexes $(M\otimes A, d_{\omega_1})\to (M\otimes A, d_{\omega_2})$. 
\end{prop}
\begin{proof}
We need to show the commutativity of the following diagram. 
\[
\xymatrixcolsep{8pc}\xymatrix{
M^i\otimes A\ar[rr]^{d_{\omega_1}=d_M\otimes \id_A+\omega_1}\ar[d]^{\exp(\lambda)}&& M^{i+1}\otimes A\ar[d]^{\exp(\lambda)}\\
M^i\otimes A\ar[rr]^{d_{\omega_2}=d_M\otimes \id_A+\overline\exp(\lambda)(\omega_1)}&&M^{i+1}\otimes A
}
\]
A direct computation reduces the commutativity to the following lemma. 
\end{proof}

\begin{lemma}
Under the above notations, for any $\xi\in M\otimes A$, the following equations hold.

\begin{equation}\label{eqLL}
\exp(\lambda)(\omega_1\xi)=(\exp(\ad \lambda)\omega_1)\exp(\lambda)\xi
\end{equation}
\begin{equation}\label{eqLLL}
\exp(\lambda)d\xi=d(\exp(\lambda)\xi)+\left(\frac{1-\exp(\ad \lambda)}{\ad \lambda}d\lambda\right)\exp(\lambda)\xi
\end{equation}
\end{lemma}

\begin{proof}[Proof of Lemma] The equation (\ref{eqLL}) is equivalent with the usual relation $e^\lam\circ\omega_1\circ e^{-\lam}=e^{[\lam,-]}\omega_1$. Let us recall the proof. We expand the right side of (\ref{eqLL}) and calculate the coefficient of term $\lambda^p \omega_1 \lambda^q \xi$. It is equal to 
\begin{align*}
\sum_{i=0}^q(-1)^{q-i}\,\frac{1}{i!}\frac{1}{(p+q-i)!}{p+q-i\choose p}&=\sum_{i=0}^q(-1)^{q-i}\,\frac{1}{i!p!(q-i)!}\\
&=(-1)^q\frac{1}{p!q!}\sum_{i=0}^q\left((-1)^{i}{p \choose i}\right)
\end{align*}
The last sum is zero unless $q=0$, and in this case, the coefficient is $\frac{1}{p!}$. This is exactly the coefficient of $\lambda^p\omega_1\xi$ on the left side of the equation. 

To show (\ref{eqLLL}), by comparing the coefficients of the term $\lambda^p(d\lambda)\lambda^q\xi$, we are lead to show the following equality,
$$\frac{1}{(p+q+1)!}=\sum_{i=0}^{q}\frac{(-1)^{p-i}}{i! (p+q-i+1)!} {p+q-i\choose p}$$ 
and this is equivalent to
$$\frac{p!q!}{(p+q+1)!}=\sum_{i=0}^q\frac{(-1)^{q-i}}{p+1+q-i}{q\choose i}.$$
Now, the right side is equal to $\int_0^1(1-t)^qt^pdt$. And by induction on $q$, we can easily show the integration is equal to $\frac{p!q!}{(p+q+1)!}$. 
\end{proof}

\begin{defn}\label{cohdef}
Let $(C,M)$ be a DGLA pair. We define $\sC^i_k(C, M; A)$ to be the full subcategory of $\sC(C; A)$ consisting the objects $\omega\in \obj\sC(C; A)$ such that $J^i_k(M\otimes_\cc A, d_\omega)=0$. This is well-defined since $(M\otimes_\cc A, d_\omega)$ is a bounded-above complex with finitely generated cohomology by Lemma \ref{lemFGCoh}.
\end{defn}

\begin{cor}
Under the notations of the previous definition, if $\sC^i_k(C, M; A)$ contains an object $\omega$ of $\sC(C; A)$, then $\sC^i_k(C, M; A)$ contains the isomorphism class of $\omega$ in $\sC(C; A)$. In other words, if $\omega\in \obj\sC^i_k(C, M; A)$, then $\overline\exp(\lambda)(\omega)\in \obj\sC^i_k(C, M; A)$ for any $\lambda\in\exp(C^0\otimes m)$. 
\end{cor}
\begin{proof}
Since $\overline{\exp}(\lambda)$ has an inverse $\overline{\exp}(-\lambda)$, Proposition \ref{propexp} implies that $(M\otimes A, d_\omega)$ is isomorphic to $(M\otimes A, d_{\omega'})$, where $\omega'=\overline\exp(\lambda)(\omega)$. Thus, for any $\lambda\in \exp(C^0\otimes m)$, $\omega\in \obj\sC^i_k(C, M; A)$ is equivalent to $\overline\exp(\lambda)(\omega)\in \obj\sC^i_k(C, M; A)$. 
\end{proof}

\begin{lemma}
Let $(C, M)$ be a DGLA pair, and let $p: A\to A'$ be a local ring homomorphism of Artinian local algebras.  For any $\omega\in \obj\sC^i_k(C, M; A)$, the image of $\omega$ under $p_*: \sC(C; A)\to\sC(C; A')$ is contained in $\obj\sC^i_k(C, M; A')$. 
\end{lemma}
\begin{proof}
Denote by $\omega'$ the image of $\omega$ under $p_*$. Since $\omega\in \obj\sC^i_k(C, M; A)$, $J^i_k(M\otimes_\cc A, d_\omega)=0$. By definition, 
$$(M\otimes_\cc A', d_{\omega'})=(M\otimes_\cc A, d_{\omega})\otimes_A A'. $$
Since $(M\otimes_\cc A, d_{\omega})$ is a complex of flat $A$-modules, by Corollary \ref{tensor}
$$J^i_k((M\otimes_\cc A, d_{\omega})\otimes_A A')=J^i_k(M\otimes_\cc A, d_{\omega})\otimes_A A'$$
and hence
$$
J^i_k(M\otimes_\cc A', d_{\omega'})=J^i_k(M\otimes_\cc A, d_{\omega})\otimes_A A'=0.
$$
Therefore, $\omega'\in \obj\sC^i_k(C, M; A')$
\end{proof}

\begin{defn}\label{cohfunctor}
The \textbf{cohomology jump functor} associated to a DGLA pair $(C, M)$ is defined to be the functor 
$$\Def^i_k(C, M): A\mapsto \iso\sC^i_k(C, M; A)$$
from $\Art$ to $\Set$. By the previous lemma, $\Def^i_k(C, M)$ is a subfunctor of $\Def(C)$. 
\end{defn}

\begin{theorem} {\bf [= Theorem \ref{independence}.]}\label{independence2}
The cohomology jump functor $\Def^i_k(C, M)$ only depends on the $i$-homotopy type of $(C, M)$. More precisely, if a morphism of DGLA pairs $g: (C, M)\to (D, N)$ is an $i$-equivalence, then the induced transformation on functors $g_*: \Def^i_k(C, M)\to \Def^i_k(D, N)$ is an isomorphism. 
\end{theorem}
\begin{proof} Given Artinian local algebra $A$, we need to show the following two conditions are equivalent for any $\omega\in \obj\sC(C; A)$,
\begin{enumerate}
\item $\omega\in \obj\sC^i_k(C, M; A)$
\item $g_{1*}(\omega)\in \obj\sC^i_k(D, N; A)$,
\end{enumerate}
where $g=(g_1, g_2)$. According to Proposition \ref{propQequiv}, it is sufficient to show that the two complexes $(M\otimes_\bC A, d_\omega)$ and $(N\otimes_\bC A, d_{g_1(\omega)})$ are $i$-equivalent. Now this follows from the argument of \cite[Theorem 3.7]{dp}: our hypothesis on $H^\ubul (g_2)$ implies that the map between the $E_1$ terms of the spectral sequences of the two complexes formed as in the proof of Lemma \ref{lemFGCoh} is an isomorphism for $\cbul\le i$ and a monomorphism for $\cbul=i+1$, and this suffices. 
\end{proof}

\section{Resonance varieties of DGLA pairs}\label{secQR}

Let $(C,M)$ be a DGLA pair. In this section we consider a nice description of $\Def(C)$ and $\Def^i_k(C,M)$ in terms of the space of flat connections and the resonance varieties which can be defined  when $(C,M)$ satisfies some finitess conditions.

\begin{defn} The {\bf space of flat connections} of $C$ is 
$$
\sF(C)=\obj\sC(C;\bC)=\{\omega\in C^1\mid d\omega+\frac{1}{2}[\om,\om]=0 \}.
$$
When $\dim C_1<\infty$, $\sF(C)$ is an affine scheme of finite type over $\bC$.
The space of flat connections of $H^\ubul(C)$ is called the {\bf quadratic cone} of $C$,  
$$\sQ(C)=\{\eta\in H^1(C)\mid [\eta,\eta]=0 \}.$$
Since by assumption $\dim H^1(C)<\infty$, $\sQ(C)$ is an affine scheme of finite type over $\bC$.
\end{defn} 

Suppose $[C^0, C^1]=0$, i.e., $[\alpha, \beta]=0$ for any $\alpha\in C^0$ and $\beta\in C^1$. Then by definition, the action of $C^0$ on $C^1$ via $\overline{\exp}$ is trivial. Thus, we have the following:

\begin{lemma}\label{lemF}
Let $C$ be a DGLA with $[C^0, C^1]=0$ and $\dim C^1<\infty$. Then $\Def(C)$ is prorepresented by $\sF(C)_{(0)}$.
\end{lemma}

\begin{cor}\label{corQ} Let $C$ be a DGLA with $[H^0(C), H^1(C)]=0$. If $C$ is 1-formal, then $\Def(C)$ is prorepresented by $\sQ(C)_{(0)}$.
\end{cor}
\begin{proof}
By Theorem \ref{gm0}, $\Def(C)$ is naturally isomorphic to $\Def(H^\ubul(C))$. The last functor is prorepresented by $\sF(H^\ubul(C))_{(0)}=\sQ(C)_{(0)}$.
\end{proof}

\begin{defn}\label{dglaresonance}
Let $(C, M)$ be a DGLA pair with  $\dim C^1<\infty$. There is a tautological section $\zeta=\zeta_{\sF(C)}$ of the sheaf $\sF(C)\otimes_\cc\mathcal{O}_{\sF(C)} $. Hence there is a universal complex on $\sF(C)$, 
$$(M\otimes_\cc \sO_{\sF(C)}, d_\zeta=d_M\otimes \id_{\sO_{\sF(C)}} +\zeta),$$
which interpolates point-wise all the complexes as in (\ref{eqAo}) with $A=\bC$.
Define the \textbf{resonance variety} $\sR^i_k(C,M)$  to be closed subscheme  of $\sF(C)$ of finite type over $\bC$ defined by the ideal $J^i_k(M\otimes_\cc \sO_{\sF(C)}, d_\zeta)$. This is well-defined as long as the complex $(M\otimes_\cc \sO_{\sF(C)}, d_\zeta)$ has finitely generated cohomology, so, in particular, when $M^i$ are finite-dimensional. The {\bf cohomology resonance variety} ${}^h\sR^i_k(C,M)=\sR^i_k(H^\ubul(C), H^\ubul(M))$ is always well-defined, and admits a presentation in terms of linear algebra: it is the subscheme of $\sQ(C)$ defined by the cohomology jump ideal $J^i_k$  of the complex $(H^\ubul(M)\otimes_\cc \sO_{\sQ(C)}, \zeta_{\sQ(C)})$, where $\zeta_{\sQ(C)}$ is the tautological section of $\sQ(C)\otimes_\cc\mathcal{O}_{\sQ(C)}$.
\end{defn}

By Corollary \ref{corFl}, we have:

\begin{lemma}\label{lemRR}
Set-theoretically, 
$$\sR^i_k(C,M)=\{\om\in \sF(C)\mid \dim H^i(M, d_\om=d_M+\omega)\geq k\} $$
when well-defined, and
$${}^h\sR^i_k(C,M)=\{\eta\in \sQ(C)\mid \dim H^i(H^\ubul(M), \eta)\geq k\}. $$
\end{lemma}

By Lemma \ref{lemF} and the definitions:
\begin{lemma}
Let $(C, M)$ be a DGLA pair with $[C^0,C^1]=0$, $\dim C^1<\infty$, and $\dim M^i<\infty$ for $i\le q$ for some $q\ge 1$. Then $\Def^i_k(C,M)$ is prorepresented by $\sR^i_k(C,M)_{(0)}$ for $i\le q$.
\end{lemma}

Hence, together with Corollary \ref{corQ} and by definitions: 
\begin{cor}\label{cohformal}
Let $(C, M)$ be a $q$-formal DGLA pair with $[H^0(C),H^1(C)]=0$, $q\ge 1$.
Then $\Def^i_k(C, M)$ is prorepresented by  ${}^h\sR^i_k(C,M)_{(0)}$ for $i\le q$.
\end{cor}

\section{Augmented DGLA pairs}\label{augment}
\begin{defn}
Let $C$ be a DGLA, and let $\mathfrak{g}$ be a Lie Algebra. We can regard $\mathfrak{g}$ as a DGLA concentrating on degree zero. An augmentation map is a DGLA map $\varepsilon: C\to \mathfrak{g}$. The augmentation ideal of $\varepsilon$ is defined to be the kernel of $\varepsilon$, which is clearly a DGLA too. Denote the augmentation ideal of $\varepsilon$ by $C_0$.  Moreover, suppose $M$ is a $C$-module. Then naturally, $M$ is also a $C_0$-module. Define the deformation functor of the augmented DGLA $(C; \varepsilon)$  by
$$
\Def(C; \varepsilon)\stackrel{\textrm{def}}{=}\Def(C_0),
$$
and the deformation functor of the augmented DGLA pair $(C, M;\varepsilon)$ by
$$
\Def(C, M; \varepsilon)\stackrel{\textrm{def}}{=}\Def(C_0, M).
$$
\end{defn}
\begin{theorem}\cite[Theorem 3.5]{gm}\label{gmaug}
Under the above notations, suppose $C$ is 1-formal, and suppose the degree zero part of $\varepsilon$, $\varepsilon^0: C^0\to \mathfrak{g}$ is surjective. Moreover, suppose the restriction of $\varepsilon^0$ to $H^0(C)$ is injective. Then $\Def(C; \varepsilon)$ is prorepresented by the formal scheme of $\sQ(C)\times \mathfrak{g}/\varepsilon(H^0(C))$ at the origin. 
\end{theorem}
We will generalize the theorem to DGLA pairs.
\begin{theorem}\label{mainaug}
Let $(C, M)$ be a DGLA pair, and let $\varepsilon: C\to \mathfrak{g}$ be an augmentation map. Suppose all the assumptions in the previous theorem hold, and moreover $(C, M)$ is $q$-formal, $q\ge 1$. Then for $i\le q$, $\Def^i_k(C, M;\varepsilon)$ is prorepresented by the formal scheme of ${}^h\sR^i_k(C,M)\times \mathfrak{g}/\varepsilon(H^0(C))$ at the origin. Furthermore, we have a commutative diagram of natural transformations of functors from $\Art$ to $\Set$,
\begin{equation}\label{comm}
\xymatrix{
\Def^i_k(C, M; \varepsilon)\ar[r]\ar[d]&({}^h\sR^i_k(C,M)\times \mathfrak{g}/\varepsilon(H^0(C)))_{(0)}\ar[d]\\
\Def(C; \varepsilon)\ar[r]&(\sQ(C)\times \mathfrak{g}/\varepsilon(H^0(C)))_{(0)}
}
\end{equation}
\end{theorem}
 
\begin{proof} This essentially follows from the previous theorem of Goldman-Millson and Theorem \ref{independence}. According to \cite[3.9]{gm}, $\Def(C; \epsilon)$ associates to every Artinian local algebra $A$ the isomorphism classes of the transformation groupoid $\sC(C; A)\bowtie \exp(\mathfrak{g}\otimes m)$, where $m$ is the maximal ideal of $A$. Recall that in Definition \ref{cat}, we defined $\sC(C; A)$ to be the transformation groupoid with objects 
$$\obj \sC(C; A)=\{\omega\in C^1\otimes_{\cc}m\;|\; d\omega+\frac{1}{2}[\omega, \omega]=0\},$$
and the morphisms are defined by the action of the nilpotent group $\exp(C^0\otimes m)$. The augmentation map induces a map of Lie groups $\exp(C^0\otimes m)\to \exp(\mathfrak{g}\otimes m)$. The objects in $\sC(C; A)\bowtie \exp(\mathfrak{g}\otimes m)$ are defined to be the Cartesian product of sets $\obj \sC(C; A)\times \exp(\mathfrak{g}\otimes m)$, and the morphisms are defined by the diagonal group action of $\exp(C^0\otimes m)$. 

By definition, $\Def^i_k(C, M; \varepsilon)$ is the subfunctor which associates to an Artinian local algebra $A$  the isomorphism classes of the transformation groupoid $\sC^i_k(C, M; A)\bowtie \exp(\mathfrak{g}\otimes m)$. Now, by \cite[Lemma 3.8]{gm}, we have an equivalence of groupoids 
$$\sC^i_k(C, M; A)\bowtie \exp(\mathfrak{g}\otimes m)\simeq\sC^i_k(H^\ubul(C), H^\ubul(M); A)\bowtie \exp(\mathfrak{g}\otimes m). $$
One can easily check that the functor $A\mapsto \iso\sC^i_k(H^\ubul(C), H^\ubul(M); A)\bowtie \exp(\mathfrak{g}\otimes m)$ from $\Art$ to $\Set$ is prorepresented by the formal scheme $({}^h\sR^i_k(C,M)\times \mathfrak{g}/\varepsilon(H^0(C)))_{(0)}$, and the diagram (\ref{comm}) commutes. 
\end{proof}

\section{Holomorphic vector bundles.} \label{holomorphic}

In this  and the next few sections we study concrete deformation problems with cohomology constraints. To a fixed setup consisting of a moduli space $\cM$ with a fixed object $\rho$ and cohomology-defined strata $\cV^i_k$ for all $i$ and $k$, we attach a DGLA pair $
(C,M)
$
such that the formal germs   $(\cM_{(\rho)},(\cV^i_k)_{(\rho)})$ prorepresent $
(\Def(C) , \Def^i_k(C,M))$  for all $i$ and $k$. We also try to find when right-hand side admits further simplifications via formality, allowing a description of the left-hand side in terms of linear algebra.

Let $X$ be a compact K\"ahler manifold. Fix $n$ and $p$. Fix a poly-stable holomorphic vector bundle, i.e. locally free $\cO_X$-module, $F$ on $X$ of any rank with vanishing Chern classes. We consider the following deformation problem with cohomology constraints:
$$
(\cM ,\rho, \cV^q_k) = (\cM(X,n), E, \cV^{pq}_k(F) ),
$$
where $\cM=\sM(X, n)$ is the moduli space  of rank $n$ stable holomorphic vector bundles on $X$ with vanishing Chern classes, and in $\sM$ one defines point-wise the Hodge cohomology jump loci with respect to $F$ to be 
\begin{equation}\label{defvb}
\sV^{pq}_k(F)=\{E\in\sM \mid \dim H^q(X, E\otimes_{\sO_X} F \otimes_{\sO_X} \Omega^p_X)\geq k \}.
\end{equation}
$\sM$ is an analytic scheme \cite{lo}.  The scheme structure of $\sV^{pq}_k(F)$ is defined locally as follows. Over a small open subset $U$ of $\sM$, there is a vector bundle $\sE$ on $X\times U$ which is locally the Kuranishi family of vector bundles. Denote by $p_1$ and $p_2$ the projections from $X\times U$ to its first and second factor. 
\begin{defn}\label{subscheme}
Locally, as a subscheme of $U$, $\sV^{pq}_k(F)$ is defined by the ideal 
$$J^q_k(\Gamma(U, \mathbf{R}p_{2*}(\sE\otimes_{p_1^{-1}\sO_X} p_1^{-1}(F\otimes_{\sO_X} \Omega_X^p)))).$$
\end{defn}

Since locally two Kuranishi families are isomorphic to each other, the subschemes patch together, and hence $\sV^{pq}_k(F)$ is a well-defined closed subscheme of $\sM$. By base change and the property of determinantal ideals, one can easily check that the closed points of $\sV^{pq}_k$ satisfy (\ref{defvb}).

\begin{defn}
For a locally free sheaf $\cF$ on $X$, denote the {Dolbeault complex of sheaves} of $\cF\otimes_{\cO_X}\Omega_X^p$ by
$$
(\sAp(\cF),\bar{\partial})\stackrel{\rm{def}}{=}(\cF\otimes_{\cO_X}\Omega_X^{p,\bullet},\bar{\partial}).
$$ The corresponding complex of global sections on $X$, which we will call the {\bf Dolbeault complex} of $\cF\otimes_{\cO_X}\Omega_X^p$, will be denoted by 
$$
(\Ap(\cF),\bar{\partial})\stackrel{\rm{def}}{=}(\Gamma(X,\cF\otimes_{\cO_X}\Omega_X^{p,\bullet}),\bar{\partial}).
$$
\end{defn}

\begin{rmk}\label{rmkDGen}
It is a standard fact (cf. \cite{gm}, \cite{M}) that the DGLA $(\Ad(\enmo(E)), \bar\partial)$ controls the deformation theory of $\sM$ at $E$. It means that the deformation functor $\Def(\Ad(\enmo(E)))$ is prorepresented by the formal scheme $\sM_{(E)}$. This is a particular case of a more general result of \cite{FIM} which states that for any complex manifold or complex algebraic variety $X$, the infinitesimal deformations of an $\cO_X$-coherent sheaf $E$ are controlled by the DGLA  of global sections $\Gamma (X,\cA^\ubul (\enmo^\ubul (\tilde{E}^\ubul)))$ of any acyclic resolution $\cA^\ubul$ of the sheaf of DGLAs $\enmo^\ubul(\tilde{E}^\ubul)$ of a locally free resolution  $\tilde{E}^\ubul$ of $E$. If $X$ is smooth, then $\cA^\ubul$ can chosen to be the Dolbeault resolution. Note that for this type of statements it does not matter if a moduli space can be constructed. Note also that to have a meaningful infinitesimal deformation problem with cohomology constraints as in  (\ref{defvb}), we must ask for $X$ to be a compact manifold or a proper algebraic variety.
\end{rmk}

 Suppose $E\in \sV^{pq}_k(F)$. Then the DGLA pair 
$$(\Ad(\enmo(E)), \Ap(E\otimes F))$$ controls the deformation theory of $\sV^{pq}_k(F)$ at $E$, where the DGLA pair structure comes from the usual map $\enmo(E)\otimes E\to E$:

\begin{thrm}\label{vb1} Let $X$ be a compact K\"ahler manifold.
For any $E\in \sM$, the natural isomorphism of functors from $\Art$ to $\Set$
$$\sM_{(E)}\cong \Def(\Ad(\enmo(E)))$$
 induces for any $p,q, k\in \nn$ a natural isomorphism of subfunctors
$$\sV^{pq}_{k}(F)_{(E)}\cong \Def^q_k(\Ad(\enmo(E)), \Ap(E\otimes F)).$$ 
\end{thrm}
\begin{proof}

Let $A$ be an Artinian local algebra. Given any $s\in \sM_{(E)}(A)=\homo(\spec(A), \sM_{(E)})$, denote its image in $\Def(\Ad(\enmo(E)))(A)$ by $\omega$. We need to show that $s\in \sV^{pq}_{k}(F)_{(E)}(A)$ if and only if $\omega\in \Def^q_k(\Ad(\enmo(E)), \Ap(E\otimes F))(A)$. As in Definition \ref{cohdef}, the complex associated to $\omega$ is $(\Ap(E\otimes F)\otimes_\cc A, d_\omega)$.

Denote by $E_s$ the pull-back of the Kuranishi family $\sE$ by the composition $\spec(A)\stackrel{s}{\to}\sM_{(E)}\to \sM$. Then $E_s$ is a vector bundle on $X_A\stackrel{\textrm{def}}{=}X\times_{\spec(\cc)} \spec(A)$.  
Denote the second projection by $p_2: X_A\to \spec(A)$. By the construction (cf. \cite[Section 6]{gm}, \cite[Proposition 3.4]{w}), $(\Ap(E\otimes F)\otimes_\cc A, d_\omega)$ is equal to the Dolbeault complex of the vector bundle $E_s\otimes_{p_1^{-1}\sO_X} p_1^{-1}(F \otimes_{\sO_X}\Omega^p_X)$, and hence it is quasi-isomorphic to $\bR p_{2*}(E_s\otimes_{p_1^{-1}\sO_X} p_1^{-1}(F \otimes_{\sO_X}\Omega^p_X))$ as complexes of $A$-modules. Therefore, 
\begin{equation}\label{idealeq1}
J^q_k(\Ap(E\otimes F)\otimes_\cc A, d_\omega)=J^q_k(\bR p_{2*}(E_s\otimes_{p_1^{-1}\sO_X} p_1^{-1}(F \otimes_{\sO_X}\Omega^p_X)))
\end{equation}
as ideals of $A$. 

Since taking determinantal ideals commutes with base change, 
\begin{equation}\label{idealeq2}
J^q_k(\bR p_{2*}(E_s\otimes_{p_1^{-1}\sO_X} p_1^{-1}(F \otimes_{\sO_X}\Omega^p_X)))=J^q_k(\bR p_{2*}(\sE\otimes_{p_1^{-1}\sO_X} p_1^{-1}(F \otimes_{\sO_X}\Omega^p_X)))\otimes_{\sO_{U}}A,
\end{equation}
where $U$ is an open subset of $\sM$ where the Kuranishi family is defined, and we use $p_1$, $p_2$ for projections to first and second factors of the products $X\times_{\spec{\cc}} \spec(A)$ and $X\times_{\spec{\cc}} U$, respectively, on each side of the equality. 

By definition, $s\in \sV^{pq}_{k}(F)_{(E)}(A)$ if and only if $$J^q_k(\bR p_{2*}(\sE\otimes_{p_1^{-1}\sO_X} p_1^{-1}F \otimes_{p_1^{-1}\sO_X} p_1^{-1}\Omega^p_X))\otimes_{\sO_{U}}A=0.$$ On the other hand, $\omega\in \Def^q_k(\Ad(\enmo(E)), \Ap(E\otimes F))(A)$ if and only if $$J^q_k(\Ap(E\otimes F)\otimes_\cc A, d_\omega)=0.$$ We obtain that $s\in \sV^{pq}_{k}(F)_{(E)}(A)$ if and only if $\omega\in \Def^q_k(\Ad(\enmo(E)),  \Ap (E\otimes F))(A)$ by (\ref{idealeq1}) and (\ref{idealeq2}).
\end{proof}

\begin{rmk}\label{rmkAbsHol}
If we replace $\cM_{(E)}$ and $\sV^{pq}_{k}(F)_{(E)}$ by the abstract deformation functors, the theorem still holds for any compact complex manifold $X$ and any holomorphic vector bundles $E$ and $F$, cf. also Remark \ref{rmkDGen}. For the purpose of this paper, we focus on the case leading to formality of the DGLA pairs. This will require the K\"ahler and the vanishing chern classes assumptions. 
\end{rmk}

\begin{que}\label{queGQVB}
One can ask a general question, in light of Remarks \ref{rmkDGen} and \ref{rmkAbsHol}: are the infinitesimal deformations of a bounded complex of $\cO_X$-coherent sheaves $E^\ubul$ on a compact complex manifold, or smooth complex algebraic variety $X$, with the hypercohomology constraint
$$
\dim_\bC \hh^q(X,E^\ubul\otimes F^\ubul\otimes \Omega_X^p)\ge k,
$$
for a bounded-above complex of $\cO_X$-coherent sheaves $F^\ubul$, governed by the DGLA pair
$$
(\Ad(\enmo^\ubul(\tilde{E}^\ubul)), \Ap (Tot^\ubul(\tilde{E}^\ubul\otimes \tilde{F}^\ubul))),
$$
where $\tilde{E}$, $\tilde{F}$ are locally free resolutions of $E$ and $F$, and $Tot^\ubul$ is the total complex? 
\end{que}

The next formality result will provide a concrete description of the formal scheme of the cohomology jump loci via linear algebra. 
\begin{thrm}[\cite{dgms}]\label{formal1}
 Let $X$ be a compact K\"ahler manifold.
For any $E\in \sM$, the DGLA pair \\
$$(\Ad(\enmo(E)), \Ap(E\otimes F))$$ is formal. 
\end{thrm}
\begin{proof}
Since both $E$ and $F$ are poly-stable and are of vanishing Chern classes, there exist flat unitary metrics on both $E$ and $F$, according to \cite{uy}. Hence $E\otimes F$ admits a flat unitary metric too. The Chern connection on $E\otimes F$ induced by the flat unitary metric is flat. Similarly, on $\enmo(E)$ there is also a flat unitary metric, whose Chern connection is also flat. Denote the $(1,0)$ part of the flat connections by $\partial$. Denote the subcomplexes of $\Ad(\enmo(E))$ and $\Ap(E\otimes F)$ consisting $\partial$-closed forms  by $K\Ad(\enmo(E))$ and $K\Ap(E\otimes F)$, respectively. Clearly, $(K\Ad(\enmo(E)), K\Ap(E\otimes F))$ is a sub-DGLA pair of $(\Ad(\enmo(E)), \Ap(E\otimes F))$, i.e., the inclusion map
\begin{equation}\label{dglamap1}
(K\Ad(\enmo(E)), K\Ap(E\otimes F))\to (\Ad(\enmo(E)), \Ap(E\otimes F))
\end{equation}
is a map of DGLA pairs. On the other hand, thanks to the existence of flat unitary metrics, $H^q(X, \enmo(E))$ can be computed by $\partial$-closed $\enmo(E)$-valued $(0, q)$-forms modulo $\partial$-exact forms, and similarly $H^q(X, E\otimes F\otimes \Omega^p)$ can be computed by $\partial$-closed $E\otimes F$-valued $(p, q)$-forms modulo $\partial$-exact forms. Hence, there is a natural surjective map of DGLA pairs. 
\begin{equation}\label{dglamap2}
(K\Ad(\enmo(E)), K\Ap(E\otimes F))\to (H^\ubul(\Ad(\enmo(E)), H^\ubul(\Ap(E\otimes F))))
\end{equation}
As in Lemma 2.2 of \cite{s1}, one can easily show that the cohomology classes of $K\Ad(\enmo(E))$ and $K\Ap(E\otimes F)$ are represented by harmonic forms. Therefore, the two maps (\ref{dglamap1}) and (\ref{dglamap2}) are both $\infty$-equivalent maps. Thus, $(\Ad(\enmo(E)), \Ap(E\otimes F))$ is formal. 
\end{proof}

\begin{rmk}\label{resonance1} Let us spell out what are the quadratic cone and the cohomology resonance varieties in this case, as defined in Section \ref{secQR}. The quadratic cone $\sQ$ of the DGLA $\Ad(\enmo(E))$ will be denoted $\sQ(E)$ and is
$$
\sQ(E)=\{\eta\in H^1(X, \enmo(E))\mid
\eta\wedge\eta=0\in H^2(X, \enmo(E))\}.
$$
The cohomology resonance variety $^h\sR^q_k(\Ad(\enmo(E)), \Ap(E\otimes F))$ will be denoted by $\sR^{pq}_k(E; F)$ to simplify the notation. Point-wise, 
$$
\sR^{pq}_k(E; F)=\{\eta\in \sQ(E) \mid  \dim H^q(H^\ubul(X, E\otimes F\otimes \Omega^p_X),\eta \wedge\cdot )\geq k\},
$$
and its scheme structure is defined using the cohomology jump ideal of the universal cohomology Aomoto complex as in Definition \ref{dglaresonance}.
\end{rmk}

It was shown by Nadel \cite{n} and Goldman-Millson \cite{gm} that there is an isomorphism of formal schemes $\cM_{(E)}\cong \sQ(E)_{(0)}$, and thus $\cM$ has quadratic algebraic singularities. The proof follows easily from Corollary \ref{corQ}. We generalize this as follows.

\begin{cor}\label{corSVB} {\bf [ = Theorem \ref{thrmHolVB}.]}
The isomorphism of formal schemes
$$
\sM_{(E)}\cong \sQ(E)_{(0)}
$$
induces for any $p,q, k\in \nn$ an isomorphism
$$
\sV^{pq}_k(F)_{(E)}\cong \sR^{pq}_k(E; F)_{(0)}. 
$$
\end{cor}
\begin{proof}
Using Yoneda's lemma, one can easily see that if two formal schemes prorepresent the same functor from $\Art$ to $\Set$, then the two formal schemes are isomorphic. Thus, the corollary is a direct consequence of Proposition \ref{formal1}, Corollary \ref{cohformal} and Proposition \ref{vb1}. The only thing we need to check is that 
\begin{equation}\label{zero}
[H^0(\Ad(\enmo(E))), H^1(\Ad(\enmo(E)))]=0.
\end{equation}
Since $E$ is stable, $H^0(\Ad(\enmo(E)))= H^0(X, \enmo(E))= \cc\cdot \id_E$. Clearly, $[\id_E, -]=0$. 
\end{proof}

\begin{rmk}
If $E$ is only poly-stable, then (\ref{zero}) is not true in general. So for the whole moduli space of semi-stable vector bundles we do not have a nice local description of the Hodge cohomology jump loci as in the above corollary. In fact, the moduli space itself may not have quadratic singularity at some points, which are semi-stable but not stable. 
\end{rmk}

\begin{cor}\label{corQuad}
Suppose $k=\dim H^q(X, E\otimes F\otimes \Omega^p_X)$. Then $\sV^{pq}_k(F)$ has  quadratic algebraic singularities at $E$. 
\end{cor}
\begin{proof}
When $k=\dim H^q(X, E\otimes F\otimes \Omega^p_X)$, the resonance variety $\sR^{pq}_k(E; F)$ is a quadratic cone in $H^1(X, \enmo(E))$. Indeed,  $\sR^{pq}_k(E; F)$ is defined by a determinantal ideal of $1\times 1$ minors, so $\sR^{pq}_k(E; F)$ is isomorphic to the intersection of the quadratic cone $\sQ(E)$ and a linear subspace. Now, it follows from the previous corollary that $\sV^{pq}_k(F)_{(E)}$ is isomorphic to the formal scheme of a quadratic cone at the origin. 
\end{proof}

\begin{rmk}
Corollary \ref{corQuad} was shown for $F\otimes\Omega_X^p=\cO_X$ by Martinengo \cite{ma} and the second author \cite{w}.
\end{rmk}

\begin{rmk}\label{rmkGL}
The case when $n=1$ and $F=\cO_X$ of the Corollary \ref{corSVB} is due to Green-Lazarsfeld \cite{gl1,gl2} and phrased in terms of their {\it derived complex}. This complex is the universal complex used by us to define cohomology resonance varieties in Definition \ref{dglaresonance}. In this case, $\cM=\Pic^\tau (X)$ is locally isomorphic via the inverse of the exponential map with the cone $\sQ(E)$ which is the whole $H^1(X,\enmo(E))=H^1(X,\cO_X)$. As in \cite{w}, by choosing $E$ to be a smooth point on the cohomology jump loci, the proof of Corollary \ref{corQuad} then implies a result in {\it loc. cit.} that $\cV^{pq}_k$ are union of translates of subtori (this has been generalized in \cite{bw}). It also implies the next corollary.
\end{rmk}

\begin{cor}\label{corRk1} {{\rm (}}\cite{gl2}, \cite[Theorem 4.2]{M-a}{\rm{)}} Assume $n=1$, that is, $\mathcal{M}=Pic^\tau(X)$. If $E$ is a singular point of $\cV^{pq}_k(F)$, then $E\in \cV^{pq}_{k+1}(F)$.
\end{cor}

\section{Representations of $\pi_1(X)$ and local systems.} \label{localsystem}

Let $X$ be finite-type CW-complex with base point $x\in X$.
Fix $n$. Let 
$$\bR(X, n)=\homo(\pi_1(X, x), GL(n, \cc)).$$ Since $\pi_1(X, x)$ is finitely presented, $\bR(X,n)$ is an algebraic scheme.

Fix $W$ a local system of any rank on $X$. We consider now the deformation problem with cohomology constraints
$$
(\bR(X,n), \rho, \ti\sV^i_k(W))
$$
where  the cohomology jump loci are defined as
$$
\ti\sV^i_k(W)=\{\rho\in \bR(X, n)\,|\, \dim H^i(X, L_\rho\otimes_\bC W)\geq k\},
$$
where $L_\rho$ is the rank $n$ local system on $X$ associated to the representation $\rho$. 

One can give $\ti\sV^i_k(W)$ a closed subscheme structure by the universal local system $\sL$ on $X\times \bR(X, n)$ as follows. Here $\sL$ is actually a local system of $R$-modules on $X$, where $R=H^0(\cO_{\bR(X,n)})$, such that $\sL\otimes_R (R/m_\rho)=L_\rho$, where $m_\rho$ is the maximal ideal of the closed point $\rho$ in $R$. Let $a:X\ra pt$ be the map from $X$ to a point. Then $Ra_*(\cL\otimes_\bC W)$ represents a bounded complex of free $R$-modules with finitely generated cohomology. Thus we can define the closed subscheme $\ti\sV^i_k(W)$ of $\bR(X,n)=\spec R$ by the ideal
$$
J^i_k(Ra_*(\cL\otimes_\bC W)).
$$
By base change and Corollary \ref{corFl}, the closed points of $\ti\sV^i_k(W)$ are the representations $\rho$ with $\dim H^i(X, L_\rho\otimes_\bC W)\geq k$. Equivalently, one can use the definition of the cohomology of local systems in terms of twisted cochain complexes on the universal covering of $X$ to define the scheme structure on $\ti\sV^i_k(W)$. The cohomology jump loci  for finite CW-complexes can be rather arbitrary \cite{w-ex}.

Assume from now that $X$ is a smooth manifold of the homotopy type of a finite CW-complex. 

\begin{defn} For a local system $\cF$ on $X$, let $(\sAr(\cF),d)$ be the {de Rham complex of sheaves} of $\cF$-valued $C^\infty$-forms on $X$. The corresponding complex of global sections on $X$, which we will call the {\bf de Rham complex} of $\cF$, will be denoted $\Ar(\cF)$. 
\end{defn}

Let $\mathfrak{g}=\enmo(L_\rho)|_{x}$, and let the DGLA augmentation map $\varepsilon: \Ar(\enmo(L_\rho))\to \mathfrak{g}$ be the restriction map. Let $\rho\in\bR(X, n)$. Goldman-Millson \cite{gm} showed that the formal scheme of $\bR(X, n)$ at $\rho$ prorepresents the functor $\Def(\Ar(\enmo(L_\rho)); \varepsilon)$. See Section \ref{augment} for the definition of this functor. We generalize this to $\ti\sV^i_k (W)$, noting first that $$(\Ar(\enmo(L_\rho)), \Ar(L_\rho\otimes_\bC W); \varepsilon)$$ is naturally an augmented DGLA pair.

\begin{thrm}\label{thmRP1} Let $X$ be a smooth manifold of the homotopy type of a finite CW-complex. The natural isomorphism \begin{equation*}\label{isoaug}
\bR(X, n)_{(\rho)}\cong \Def(\Ar(\enmo(L_\rho)); \varepsilon).
\end{equation*}
induces for any $i,k\in \nn$ a natural isomorphism of subfunctors,
$$
(\ti\sV^i_k(W))_{(\rho)}\cong \Def^i_k(\Ar(\enmo(L_\rho)), \Ar(L_\rho\otimes_\bC W); \varepsilon).
$$
\end{thrm}
\begin{proof} This is similar to the proof of Theorem \ref{vb1}. Let $A$ be an Artinian local algebra. Given any $s:\spec A\ra \bR(X,n)_{(\rho)}$, denote its image in $\Def(\Ar(\enmo(L_\rho));\varepsilon)$ by $\omega$. Let $\cL_s$ be the induced $A$-local system $\cL\otimes_R A$ on $X$. Then $(\Ar(L_\rho\otimes_\bC W)\otimes_\bC A,d_\omega)$ is the de Rham complex of the $A$-local system $\cL_s\otimes_\bC W$ on $X$ (cf. \cite[Section 6]{gm}). Thus it is quasi-isomorphic with $Ra_*(\cL_s\otimes_\bC W)$ as complexes of $A$-modules.  So
$$
J^i_k(\Ar(L_\rho\otimes_\bC W)\otimes_\bC A,d_\omega)=J^i_k(Ra_*(\cL_s\otimes_\bC W)),
$$
which in turn equals $J^i_k(Ra_*(\cL\otimes_\bC W))\otimes_R A$ by Corollary \ref{tensor}.
\end{proof}

\begin{rmk} 
This theorem generalizes a result of Dimca-Papadima \cite{dp} who proved it for the reduced structure of the cohomology jump loci at the trivial representation, that is, for the germ of $(\ti\sV^i_k)^{red}$ at $\bone$. In \cite{dp}, $X$ is allowed to be a connected CW-complex of finite type by replacing the de Rham complex with Sullivan's  de Rham complex of piecewise $C^\infty$ forms. For simplicity, we opted to leave out this topological refinement.
\end{rmk}

Along with representations of the fundamental group, let us consider the closely-related deformation problem for the associated local systems. The relation at the level of deformations between representations (i.e. local systems with a frame at a fixed point) and local systems is a particular case of the relation between the deformation functors of an augmented DGLA pair and those of the DGLA pair itself, see Theorem \ref{mainaug}.

For now, the assumptions are the same: $X$ is a smooth manifold of the homotopy type of a finite CW-complex and $W$ is a local system on $X$. We consider the deformation problem with cohomology constraints:
$$
(\mb=\cM(X,n), L, \cV^{i}_k(W)),
$$
where $\mb=\mb(X, n)$ is the moduli space of irreducible rank $n$ local systems on $X$ and 
$$
\cV^i_k(W)=\{ L\in \mb\mid \dim_\bC H^i(X,L\otimes_\bC W)\ge k \}.
$$
The natural subscheme structure of $\cV^i_k(W)$ in $\mb$ is defined as follows.  $GL(n, \cc)$ acts on $\bR(X, n)$ by conjugation. Clearly these actions preserve all the cohomology jump loci $\ti\sV^i_k(W)$ of representations. Since $\mb$ is an open subset of the GIT quotient of $\bR(X, n)$ by $GL(n, \cc)$, $\sV^i_k(W)$ can be defined as the the intersection of $\mb$ and the image of $\tilde\sV^i_k(W)$ under the GIT quotient map. 

The argument in Section \ref{holomorphic} works similarly for  moduli spaces of local systems. Since the proofs are essentially the same, we only state the results. 

Let $L$ be in $\mb$. Then $(\Ar(\enmo(L)), \Ar(L\otimes W))$ is naturally a DGLA pair. It is a standard fact that the deformation functor $\Def(\Ar(\enmo(L)))$ is prorepresented by the formal scheme $(\mb)_{(L)}$.

\begin{thrm} Let $X$ be a smooth manifold of the homotopy type of a finite CW-complex.
The natural isomorphism of  functors $$(\mb)_{(L)}\cong\Def(\Ar(\enmo(L)))$$ induces for any $i, k\in \nn$ a natural isomorphism of subfunctors $$\sV^i_k(W)_{(L)}\cong\Def^i_k(\Ar(\enmo(L)), \Ar(L\otimes W)).$$ 
\end{thrm} 

In this last result, the condition of irreducibility of the local system can be removed if we replace $(\mb)_{(L)}$ and $\sV^i_k(W)_{(L)}$ by the abstract deformation functors, cf. Remark \ref{rmkAbsHol}. However, we are again focusing on the case leading to formality, for which at least a semi-simplicity condition is crucial. Irreducibility will be used to further simplify the answer of the deformation problem in terms of resonance varieties.

\begin{thrm}\label{lsformal} Let $X$ be a compact K\"{a}hler manifold, $L\in \mb$, and let $W$ be a semi-simple local system on $X$.
Then the DGLA pair $(\Ar(\enmo(L)), \Ar(L\otimes W))$ is formal. 
\end{thrm}
\begin{proof}
The proof  is essentially the same as the proof of the Theorem \ref{formal1}, except here we need to use the harmonic metric on the flat bundle in the sense of \cite{s1} instead of the flat unitary metric before. 
\end{proof}

In the situation of Theorem \ref{lsformal}, as in Remark \ref{resonance1}, the quadratic cone of $\Ar(\enmo(L))$ is
$$
\sQ(L)=\{\eta\in H^1(X, \enmo(L)) \mid 
\eta\wedge\eta=0\in H^2(X, \enmo(L))\},
$$
and the cohomology resonance varieties of the DGLA pair are
$$
\sR^i_k(L;W)=\{ \eta\in \sQ(L)\mid \dim H^i(H^\ubul(X, L\otimes W), \eta\wedge \cdot)\geq k \},
$$
with the scheme structure of $\sR^i_k(L;W)$ defined using the universal Aomoto complex, as in Definition \ref{dglaresonance}. The condition on the irreducibility of $L$, as opposed to just semi-simplicity, is now used to derive the analog of Corollary \ref{corSVB} for local systems with a similar proof:

\begin{cor}\label{corVLS}{\bf [ = Theorem \ref{thmIrrLS}.]} Let $X$ be a compact K\"{a}hler manifold, $L\in \mb$, and let $W$ be a semi-simple local system on $X$.
The isomorphism of formal schemes
$$
(\mb)_{(L)}\cong \sQ(L)_{(0)}
$$
induces for any $i, k\in \nn$ an isomorphism
$$
\sV^i_k(W)_{(L)}\cong\, \sR^i_k(L;W)_{(0)}. 
$$
\end{cor}

The analogs of the Corollaries \ref{corQuad} and \ref{corRk1} also hold.

\begin{rmk}\label{rmkPS}
Corollary \ref{corVLS} for rank one local systems $E$ and $W=\bC_X$ also follows from the strong linearity theorem of Popa-Schnell namely \cite[Theorem 3.7]{PS}. In fact, our approach gives a different proof of the strong linearity theorem, at least of the fact that the two complexes appeared in \cite[Theorem 3.7]{PS} are quasi-isomorphic (in the derived category) after restricting to the formal neighborhood of the origin. One can argue as follows. For an Artinian local algebra $A$ and a map from $\spec(A)$ to the formal neighborhood, one can restrict the two complexes in \cite[Theorem 3.7]{PS} to $\spec(A)$. After the restriction, the two complexes can be connected to another one via a zig-zag  using the proof of Theorem \ref{independence2} and the proof of Theorem \ref{formal1}. The two maps in the zig-zag are quasi-isomorphisms. Since the zig-zag is canonical, it allows us to take inverse limit for all such $A$. After taking limit, we obtain two quasi-isomorphisms which connect the two complexes on the formal neighborhood of origin. Note that the proof of \cite[Theorem 3.7]{PS} gives a stronger statement: the quasi-isomorphism is obtained by one single map. The proof we sketched gives that the quasi-isomorphism is obtained by one zig-zag. However, this suffices for the application to cohomology jump loci.
\end{rmk}

Now, coming back to representations, by Theorem \ref{mainaug} and Theorem \ref{lsformal} we have the following:

\begin{cor}\label{rep1} Let $X$ be a compact K\"ahler manifold. Let $\rho\in \bR(X, n)$ be a semi-simple representation, and let $W$ be a semi-simple local system. Then there is an isomorphism of formal schemes
$$\ti\sV^i_k(W)_{(\rho)}\cong (\sR^i_k(L_\rho,W)\times \mathfrak{g}/\mathfrak{h})_{(0)}$$
where  $\mathfrak{h}=\varepsilon(H^0(X,\enmo (L_\rho)))$.
\end{cor}
\begin{proof}
To apply Theorem \ref{mainaug}, we only need to check the assumptions in Theorem \ref{gmaug}. It is obvious that $A^0_{\textrm{DR}}(\enmo(L_\rho))\to \enmo(L_\rho)|_{x}$ is surjective. Since $\rho$ is semi-simple, or equivalently $L_\rho$ is a semi-simple local system, we can assume it splits into a direct sum of simple local systems $L_\rho=\bigoplus_{j\in J} L_j$. Then $H^0(\Ar(\enmo(L_\rho)))$ is generated by $\id_{L_j}$, $j\in J$. Therefore, $\epsilon^0: H^0(\Ar(\enmo(L_\rho)))\to \enmo(L_\rho)|_{x}$ is injective. 
\end{proof}

We can give another equivalent description of the cohomology resonance variety $\sR^i_k(L_\rho,W)$ and the affine space $\mathfrak{g}/\mathfrak{h}$. It is well-known that the tangent space of $\bR(X, n)$ at the point $\rho$ is isomorphic to the vector space of 1-cocycles $Z^1(\pi_1(X), \mathfrak{gl}(n, \cc)_{\ad \rho})$, see \cite{gm}. Moreover, we have the following isomorphism,
\begin{equation}\label{eqZ1}
Z^1(\pi_1(X), \mathfrak{gl}(n, \cc)_{\ad \rho})/B^1(\pi_1(X), \mathfrak{gl}(n, \cc)_{\ad \rho})\cong H^1(X, \enmo(L_\rho)).
\end{equation}
In fact, one can easily check that $B^1(\pi_1(X), \mathfrak{gl}(n, \cc)_{\ad \rho})\cong \mathfrak{g}/\mathfrak{h}$.  For 1-cocycle 
$\eta$ in the vector space $Z^1(\pi_1(X), \mathfrak{gl}(n, \cc)_{\ad \rho})$, denote the image in $H^1(X, \enmo(L_\rho))$ under the above isomorphism by $\bar\eta$. 

\begin{defn}\label{defnQrho}
Define the \textbf{quadratic cone} of $\rho$ to be
$$\sQ(\rho)=\{\eta\in Z^1(\pi_1(X), \mathfrak{gl}(n, \cc)_{\ad \rho})\,|\,\bar\eta\wedge\bar\eta=0 \in H^2(X, \enmo(L_\rho))\}.$$
Define the \textbf{twisted resonance varieties} of $\rho$ to be 
$$
\sR^i_k(\rho,W)=\{\eta\in \sQ(\rho)\,|\, \dim H^i(H^\ubul(X, L_\rho\otimes_\bC W), \bar\eta\wedge\cdot)\geq k\}. 
$$
As in Definition \ref{subscheme}, using the universal family, we can give $\sR^i_k(\rho,W)$ a subscheme structure. 
\end{defn}

Simpson \cite{s1} showed that there is an isomorphism of formal schemes
$
\mathbf{R}(X, n)_{(\rho)}\cong \sQ(\rho)_{(0)}
$ for a semi-simple representation $\rho$. We generalize this to $\ti\sV^i_k(W)_{(\rho)}$. First, we need the following:

\begin{lemma}\label{rep2} Let $X$ be a compact K\"ahler manifold.
There is a non-canonical isomorphism of schemes
$$
H^1(X,\enmo(L_\rho))\times \mathfrak{g}/\mathfrak{h}\cong Z^1(\pi_1(X), \mathfrak{gl}(n, \cc)_{\ad \rho}).
$$
This induces an isomorphism of subshemes
$$
\sR^i_k(L_\rho,W)\times \mathfrak{g}/\mathfrak{h}\cong \sR^i_k(\rho,W)
$$
if $\rho$ and $W$ are semi-simple.
\end{lemma}
\begin{proof}
The first claim follows from  (\ref{eqZ1}) and the remark after. Now $B^1(\pi_1(X), \mathfrak{gl}(n, \cc)_{\ad \rho})$ acts on $Z^1(\pi_1(X), \mathfrak{gl}(n, \cc)_{\ad \rho})$. By definition, $ \sR^i_k(\rho,W)$ is invariant under this action. Therefore, $ \sR^i_k(\rho,W)$ is equal to the pull-back of some closed subscheme $\bar \sR^i_k(\rho,W)$ of the quotient
$Z^1(\pi_1(X), \mathfrak{gl}(n, \cc)_{\ad \rho})/B^1(\pi_1(X), \mathfrak{gl}(n, \cc)_{\ad \rho})$. Under the isomorphism (\ref{eqZ1}) one can easily see that $\sR^i_k(L_\rho,W)$ and $\bar \sR^i_k(\rho,W)$ are defined by the same universal  complexes, and hence they are isomorphic. The conclusion follows.
\end{proof}

From Corollary \ref{rep1} and Lemma \ref{rep2} we get:
\begin{theorem}\label{mainkahler2}{\bf [= Theorem \ref{thmRPP}.]} Let $X$ be a compact K\"ahler manifold.
Let $\rho\in \bR(X, n)$ be a semi-simple representation, and let $W$ be a semi-simple local system. Then the isomorphism of formal schemes 
$$
\mathbf{R}(X, n)_{(\rho)}\cong \sQ(\rho)_{(0)}
$$ induces an isomorphism
$$
\ti\sV^i_k(W)_{(\rho)}\cong \sR^i_k(\rho,W)_{(0)}. 
$$
\end{theorem}
When $k=\dim H^i(X, L_\rho)$, the resonance variety $\sR^i_k(\rho,W)$ is equal to the intersection of the quadratic cone $\sQ(\rho)$ and a linear subspace of $Z^1(\pi_1(X), \mathfrak{gl}(n, \mathbb{C})_{\textrm{ad}\rho})$, see the proof of Corollary \ref{corQuad}. Hence, $\sR^i_k(\rho,W)$ is also a quadratic cone. Thus we have the following corollary.

\begin{cor}\label{corQRep} Let $X$ be a compact K\"ahler manifold.
Let $\rho\in \bR(X, n)$ be a semi-simple representation,  and let $W$ be a semi-simple local system. Suppose $k=\dim H^i(X, L_\rho)$. Then $\ti\sV^i_k(W)$ has quadratic singularities at $\rho$. 
\end{cor}

\section{Stable Higgs bundles}\label{secHB}
According to nonabelian Hodge theory due to Simpson, given a smooth projective complex variety $X$, one can consider three moduli spaces and the cohomology jump loci in them: $\mb(X, n)$, $\mdr(X, n)$, $\mdol(X, n)$, denoting the moduli spaces of irreducible local systems of rank $n$, stable flat bundles of rank $n$, and stable Higgs bundles with vanishing Chern classes of rank $n$, respectively, see \cite{s3}. Although $\mb (X,n)$ can be constructed for any topological space with finitely generated fundamental group, the assumption that $X$ is smooth projective is essential for the construction of $\mdr(X, n)$ and $\mdol(X, n)$. Since $\mb(X, n)$ and $\mdr(X, n)$ are isomorphic as analytic spaces, and since the isomorphism induces isomorphisms on the cohomology jump loci, the deformation problems with cohomology constraints are same for irreducible local systems and stable flat bundles. 

We consider now the deformation problem with cohomology constraints
$$
(\mdol=\mdol(X,n), E=(E,\theta), \sV^i_{k}(F))
$$
where $F=(F,\phi)$ is a poly-stable Higgs bundle with vanishing Chern classes and
\begin{align*}
\sV^i_{k}(F) =\{  (E, \theta)  \in \mdol\mid \dim\hh^i(X, (E\otimes_{\cO_X}F\otimes_{\cO_X} \Omega^\ubul_X, \theta\otimes 1+1\otimes\phi))\ge k\}.
\end{align*}
The subscheme structure of $\sV^i_{k}(F)$ is defined as follows. Fix a base point $x\in X$. $\mdol$ is the GIT quotient by $GL(n, \cc)$ of a fine moduli space $\bR_{\textrm{Dol}}(X, x, n)$ of rank $n$ stable Higgs bundles $(E, \theta)$ on $X$ with vanishing Chern classes together with a basis $\beta: E|_{x}\to \cc^n$. On $\bR_{\textrm{Dol}}(X, x, n)$, there is a universal family of Higgs bundles. Using this universal Higgs bundles, we can define cohomology jump loci in $\bR_{\textrm{Dol}}(X, x, n)$ as a closed subschemes. These cohomology jump loci are invariant under the $GL(n, \cc)$ action. Thus we can define their image under the quotient map to be $\sV^i_{k}(F)$, which has a closed subscheme structure. 

This deformation problem with cohomology constraints is parallel to the case of irreducible local systems. We will only state the main theorem. We leave all the statements and the proofs of the other corollaries to the reader. 

\begin{defn}
For a Higgs bundle $(\cF, \psi)$ we define the {\bf Higgs complex} as the complex of global $\cF$-valued $C^\infty$-forms with differential $\bar\partial+\psi$. We denote this complex by $(\Ah (\cF),\bar\partial+\psi)$, or simply $\Ah(\cF)$.
\end{defn}

For a Higgs bundle $(E, \theta)$ in $\mdol$, we have a DGLA pair
$$
(\Ah(\enmo(E)), \Ah(E\otimes_{\cO_X}F)),
$$
where the Higgs field on the locally free $\cO_X$-module $\enmo(E)$ is 
$[\theta,\cdot]$, and the Higgs field on $E\otimes _{\cO_X}F$ is as in the complex from the definition of $\sV^i_k(F)$. The standard fact is that the formal scheme of $\mdol$ at ${(E, \theta)}$ prorepresents the functor $\Def(\Ah(\enmo(E)))$, see \cite{Mart}.

\begin{thrm}
The natural isomorphism of functors $$(\mdol)_{(E, \theta)}\cong\Def(\Ah(\enmo(E)))$$ induces for any $i, k\in \nn$ a natural isomorphism of subfunctors $$(\sV^i_{k}(F))_{(E, \theta)}=\Def^i_k(\Ah(\enmo(E)), \Ah(E\otimes F)).$$ The DGLA pair $(\Ah(\enmo(E)), \Ah(E\otimes F))$ is formal, and hence its quadratic cone and cohomology resonance variety determine the formal germs at $(E,\theta)$ of $\mdol$ and $\sV^i_k(F)$.
\end{thrm}

\section{Other consequences of formality}\label{secIneq}

In this section we point out how the formality of a DGLA pair $(C,M)$ has implications on the possible shape of the Betti numbers of $M$ and on the geometry of the cohomology resonance varieties ${}^h\sR^i_k(C,M)$.

Let $(C,M)$ be a formal DGLA pair. We will use the following simplifying notation in this section:
\begin{align*}
Q&=Q(C), \\
\sR^i_k &= {}^h\sR^i_k(C,M),\\
\pp &=\pp(H^1(C)),\\
b_i& =\dim_\cc H^i(M).
\end{align*}
Let $R$ be the homogeneous coordinate ring of the projectivization $\pp Q$ of the quadratic cone $Q$ in $\pp$. Consider the universal complex from Definition \ref{dglaresonance} on $\pp Q$
\begin{equation}\label{eqSC}
H^0(M)\otimes_\cc\sO_{\pp Q}\mathop{\lra}^{\zeta_Q}
\ldots\lra H^k(M)\otimes_{\cc}\sO_{\pp Q}(k) \mathop{\lra}^{\zeta_{Q}}\ldots
\end{equation}
and the associated complex of graded $R$-modules
\begin{equation}\label{eqGC}
H^0(M)\otimes_\cc R\mathop{\lra}^{\zeta_Q}
\ldots\lra H^k(M)\otimes_{\cc}R(k) \mathop{\lra}^{\zeta_{Q}}\ldots
\end{equation}
By definition, multiplication with $\zeta_Q$ are graded maps of degree one, hence the shifts. The cohomology jump ideals $J^i_k$ of these complexes define $\bP\sR^i_k$ inside $\bP Q$.  Let  
$$a=a(C,M)\stackrel{\textrm{def}}{=}\min \{i\mid H^i(H^\ubul(M)\otimes_{\cc}R(_{^\ubul}),\zeta_Q)\ne 0\}$$
measure how far to the right the complex (\ref{eqGC}) is exact. Therefore the complexes
\begin{equation}\label{eqSC2}
H^0(M)\otimes_\cc\sO_{\pp Q}(-a)\mathop{\lra}
\ldots\lra H^a(M)\otimes_{\cc}\sO_{\pp Q} 
\end{equation}
and 
\begin{equation}\label{eqGC2}
H^0(M)\otimes_\cc R(-a)\mathop{\lra}
\ldots\lra H^a(M)\otimes_{\cc}R 
\end{equation}
are exact except in degree $a$, and the complex (\ref{eqGC2}) is a minimal graded free resolution of the cokernel of the last map. We will call $\phi_i$ the maps in these complexes from the degree $i$ term to the degree $i+1$ term.

There are two sources of restrictions on the possible Betti numbers $b_i$ and on the geometry of the resonance varieties $\pp\cR^i_k$ for $i<a$: one from the Chern classes of the vector bundles in (\ref{eqSC2}), and another one from the relation of $\bP\cR^i_k$ with Fitting ideals of the maps in (\ref{eqGC2}). The Chern classes technique was first applied by Lazarsfeld-Popa \cite{LP} to, in our language, the DGLA pair controlling the infinitesimal deformations of $\cO_X$ in $\Pic^\tau(X)=\cM(X,1)$ with cohomology constraints when $X$ is a compact K\"ahler manifold, see Section \ref{holomorphic}. Fitting ideals were also used by Fulton-Lazarsfeld to prove connectedness results for Brill-Noether loci, which are particular cases of cohomology jump loci. For applications of Fitting ideals to twisted higher-rank Brill-Noether loci, see the survey \cite{TiB}. It was noticed in \cite{B-h} that the case of the trivial local system of rank one on the complement of a hyperplane arrangement is similar, where the Chern classes approach and the relation with Fitting ideals were also explored. This similarity is explained and generalized in this section by observing that we can run the arguments for any formal DGLA pair.

The following results were stated in \cite{B-h} for hyperplane arrangement complements. However, in that case $\bP Q=\bP$, which is not true in general. 

\begin{prop}\label{propIneqs}
With notation as above for a formal DGLA pair $(C,M)$, let $i<a$. 

(a) Let $\beta_i=\rank (\phi_i)$ and let $I_{\beta_i}(\phi_i)$ be the ideal in $R$ generated by the minors of rank $\beta_i$ of $\phi_i$. Then $b_i=\beta_i+\beta_{i-1}$ and ${\rm{depth}}(I_{\beta_i}(\phi_i))\ge a-i$.

(b) $(\bP\cR^i_1)^{red}$ is the  support of $I_{\beta_i}(\phi_i)$.

(c) $(\bP\cR^{i-1}_1)^{red}\subset (\bP\cR^i_1)^{red}$.

(d) $\codim_{\bP Q} \bP\cR^i_1 \ge a-i$ if $R$ is Cohen-Macaulay.

(e) $\codim_{\bP Q} \bP\cR^i_1\le (\beta_{i-1}+1)(\beta_{i+1}+1)$ if $R$ is Cohen-Macaulay.

(f) $\cR^0_k$ is defined by $I_{\beta_0+1-k}(\phi_i)$.

(g) $(\bP\cR^i_k)^{red}$ contains the  support of $I_{\beta_i+1-k}(\phi_i)$, and equals it away from $(\bP\cR^i_1)$. 

(h) $\codim_{\bP Q} \bP\cR^i_k\le (\beta_{i-1}+k)(\beta_{i+1}+k)$ if $R$ is Cohen-Macaulay.

(i) $\bP\cR^i_k$ is connected away from the components of $\bP\cR^i_1$ which are disconnected from the support of $I_{\beta_i+1-k}(\phi_i)$, if $\bP Q$ is irreducible and reduced. 

(j) $(\cR^{i}_1)^{red}\subset (\cR^{i+1}_2)^{red}$. If $i<a-1$ and $k\le 1+\frac{a-2}{i+1}$, then $(\cR^i_1)^{red}\subset (\cR^{i+1}_j)^{red}$.

(k) $b_i\ge \binom{a}{i}$ if $R$ is a polynomial ring.   $\beta_i\ge a-i$ if $R$ is Cohen-Macaulay.

Let $q_i=\codim_{\bP Q}\bP \cR^i_1$, and for $i>0$ let $$c_t^{(i)}=\prod_{k=1}^{i+1}(1-k\cdot t)^{(-1)^kb_i+1-k}.$$ Let $c_j^{(i)}$ be the coefficient of $t^j$ in $c_t^{(i)}$. Assume that $\chi_a(M):=b_a-b_{a-1}+b_{a-2}-\ldots\ne 0$. 
 
(l)  Any Schur polynomial of weight $< q_i$ in $c_1^{(i)},\ldots, c_{q_i-1}^{(i)}$ is non-negative.
\end{prop}

\begin{proof}
(a) This is \cite[Theorem 20.9]{eisenbud}.

(b) The proof is essentially the same as for \cite[Proposition 3.4]{B-h}. By truncating (\ref{eqGC2}) and repeating the following argument, it is enough to show only the case $i=a-1$. Using the complex of sheaves (\ref{eqSC2}), let $\sF={\rm{coker}}(\phi_{a-1})$. The support of the Fitting ideal $I_{\beta_{a-1}}(\phi_{a-1})$ is the locus of closed points in $\bP Q$ where $\sF$ fails to be locally free. The claim follows now from Lemma \ref{lemRR}, Lemma \ref{lemEPY}, and the fact that $\Tor_1^{\cO_{\bP Q,\bar{\eta}}}(\kappa(\bar{\eta}),\cF_{\bar{\eta}})=0$ iff the stalk $\cF_{\bar{\eta}}$ is free.

(c) Follows from (b) and \cite[Corollary 20.12]{eisenbud}.

(d) Follows from (a) and (b), since the Cohen-Macaulay condition implies that depth equals codimension.

(e) See \cite[Theorem 1.2]{B-h}. The result of Eagon-Northcott used there holds if $R$ is Cohen-Macaulay.

(f) It follows by definition.

(g) This is essentially the proof of Theorem 1.1 from \cite{B-h} and its Erratum. Again, it is enough to prove the case $i=a-1$. By Lemma \ref{lemRR} and Lemma \ref{lemEPY},
$$
(\bP\cR^{a-1}_k)^{red}=\{\bar{\eta}\in \bP Q\mid \Tor_1^{\cO_{\bP Q,\bar{\eta}}}(\kappa(\bar{\eta}),\cF_{\bar{\eta}})\ge k\}.
$$
The support of the Fitting ideal $I_{\beta_{a-1}+1-k}(\phi_{a-1})$ is
$$
\{\bar{\eta}\in\bP Q\mid m(\cF_{\bar{\eta}})-\rank(\cF_{\bar{\eta}})\ge k\},
$$
where $m(\cF_{\bar{\eta}})$ is the minimal number of generators of $\cF_{\bar{\eta}}$ over $\cO_{\bP Q,\bar{\eta}}$. The rank is well-defined since minimal free resolutions exist over local rings, and by the characterization of exactness of a complex from \cite[Theorem 20.9]{eisenbud}. Thus we do not need to assume that $\cO_{\bP Q,\bar{\eta}}$ is a domain as in {\it loc. cit.}. The rest of the argument is as in {\it loc. cit.}

(h) Follows from (f) as in the proof of (e).

(i) Follows from (f) and from the Fulton-Lazarsfeld connectedness theorem, see Erratum, Corollary 1.2 of \cite{B-h}.

(j) See \cite[Corollary 1.3]{B-h}.

(k) See \cite[Proposition 3.2]{B-h}. The result of Herzog-K\"uhl used holds for the case when $R$ is a polynomial ring. The result of Evans-Griffiths used holds for the case when $R$ is Cohen-Macaulay.

(l) This is essentially the same proof as for \cite[Theorem 3.1]{B-h}. Consider the case $i=a-1$ firstly. By (b), $\bP\cR^{a-1}_1$ is the locus of points in $\bP Q$ where $\cF$ fails to be locally free. Let $W$ be a generic vector subspace of $H^1(C)$ of codimension $\dim \bP\cR^{a-1}_1$+1. Then the restriction of (\ref{eqSC2}) to $X=\bP Q\cap\bP W$ gives an exact sequence of locally free sheaves on $X$:
\begin{equation}\label{eqRes}
0\ra H^0M\otimes\cO_{X}(-a)\ra\ldots\ra H^aM\otimes\cO_{X}\ra \cF_{| X}\ra 0.
\end{equation}
Since { we assume} $\chi_a(M)\ne 0$, the restriction of $\cF$ to $X$ is non-zero. Moreover, this is a globally generated vector bundle, so Fulton-Lazarsfeld positivity applies, see \cite[12.1.7 (a)]{Fu}: for any positive $k$-cycle $\alpha$ on $X$, the intersection $P_j\cap \alpha$ is the rational equivalence class of a non-negatively supported $k-j$ cycle on $X$, where $P_j$ is any Schur polynomial of weight $j\le k$ in the Chern classes of $\cF_{|X}$.

Let $i:X\ra \bP$ be the natural inclusion. For a vector bundle $E$ on $\bP$ and for $\alpha\in A_*(X)$, the projection formula says that $i_*(c_j(i^* E)\cap\alpha)=c_j(E)\cap i_*\alpha$. From (\ref{eqRes}) it is not difficult to see that the same holds for $\cF_{|X}$, namely 
\begin{equation}\label{eqInt}
i_*(c_j(\cF_{|X})\cap\alpha) =C_j\left\{\prod_{k=1}^a(1+c_1(\cO_{\bP}(-k))\cdot t)^{(-1)^k b_{a-k}}\right\}\cap i_*\alpha,
\end{equation}
where $C_j$ stands for the coefficient of $t^j$. Let $\alpha=[X]\cap [L]\in A_k(X)$ with $[L]\in A_*(\bP)$ be the class of a linear section. Then the degree of $i_*\alpha $ is $\deg (\bP Q)$. Thus the non-negativity result of Fulton-Lazarsfeld implies that 
$$
C_j\left\{\prod_{k=1}^a(1-k\cdot t)^{(-1)^k b_{a-k}}\right\}\cdot \deg(\bP Q)\ge 0
$$
for $j\le k$, and so for $j\le \dim X=q_{a-1}-1$. Thus the claim follows in this case for $P_j=c_j(\cF_{|X})$. For the other Schur polynomials, a repeated application of (\ref{eqInt}) reduces the claim to this case. 

For the case $i<a-1$, note that $\chi_{i}(M)\ne 0$ for $i<a$ by (c) and the assumption that $\chi_a(M)\ne 0$. Hence this case follows by the same argument by truncating (\ref{eqRes}) and shifting to get global generation. 
\end{proof}

The following, which was used above, was proved for the case $\bP Q=\bP$ in \cite[Theorem 4.1]{EPY} using the BGG correspondence.

\begin{lm}\label{lemEPY} With the notation as above, let $\cF={\rm{coker}}(\phi_{a-1})$ in (\ref{eqSC2}). Let $0\ne\eta\in Q$ and denote its image in $\bP Q$ by $\bar{\eta}$. Then for $i\ge 0$
$$
H^{a-i}(H^\ubul M,\eta .)=\Tor_i^{\cO_{\bP Q,\bar{\eta}}}(\kappa (\bar{\eta}),\cF_{\bar{\eta}}),
$$
where $\kappa(\bar{\eta})$ is the residue field of $\eta$, and $\cF_{\bar{\eta}}$ is the stalk of the sheaf $\cF$ at $\bar{\eta}$.
\end{lm}
\begin{proof}
By induction on $k$, we can assume that
$$
H^{k-i}\sigma_{\le k}(H^\ubul M,\eta .)=\Tor_i^{\cO_{\bP Q,\bar{\eta}}}(\kappa(\bar{\eta}), \cK_{k-1,\bar{\eta}})
$$
for $i\ge 0, k<a$, where $\sigma_{\le k}$ is the stupid truncation and $\cK_k={\rm{coker}}(\phi_{k})$ in (\ref{eqSC2}). By applying $.\otimes_{\cO_{\bP Q,\bar{\eta}}}\kappa(\bar{\eta})$ to the exact sequence
$$
0\lra \cK_{a-2} \lra H^a(M)\otimes\cO_{\bP Q}\lra \cK_{a-1}\lra 0,
$$
we obtain that 
$$
\Tor_{i}^{\cO_{\bP Q,\bar{\eta}}}(\kappa(\bar{\eta}),\cK_{a-1,\bar{\eta}})= \Tor_{i-1}^{\cO_{\bP Q,\bar{\eta}}}(\kappa({\bar{\eta}}),\cK_{a-2,\bar{\eta}})=H^{a-i}(H^\ubul M,\eta .)
$$
for $i\ge 2$. Since the case $i=0$ is obvious, it remains to prove the case $i=1$. This case follows since the map $H^{a-1}(M)\ra H^a(M)$ decomposes via the surjection $H^{a-1}(M)\ra \cK_{a-2,\bar{\eta}}\otimes \kappa(\bar{\eta})$, and we have an exact sequence
$$
0\lra \Tor_{1}^{\cO_{\bP Q,\bar{\eta}}}(\kappa(\bar{\eta}),\cK_{a-1,\bar{\eta}})\lra \cK_{a-2,\bar{\eta}}\otimes\kappa(\bar{\eta}) \lra H^a(M)\lra \cK_{a-1}\otimes\kappa(\bar{\eta}) \lra 0. $$
\end{proof}

\begin{rmk} Given a particular deformation problem with cohomology constraints, it is interesting to determine geometrically the number $a=a(C,M)$ for a DGLA pair governing the deformation problem. Let us give some examples.

(a) Consider the formal DGLA pair $(\Ad (\enmo(\cO_X)),\Ap (\cO_X))$ governing the deformation problem with cohomology constraints $(\Pic^\tau(X), \cO_X, \cV^{pq}_k(\cO_X))$ from Remark \ref{rmkGL}. Then Proposition \ref{propIneqs} becomes a result about $\cV^{pq}_k(\cO_X)$ and $h^q(X,\Omega_X^p)$ via Corollary \ref{corSVB}. In this case the deformation problem can be stated on the Albanese of $X$, but one pays the price that one has to know something about the Albanese map. For  $p=0$ or $\dim X$,
$$
a=\dim X-\dim (\text{generic fiber of the Albanese map}),
$$
see \cite{LP} where the statements on the numbers $h^q(\cO_X)$ are also proven. For other values of $p$, it is enough to consider $p+q\le\dim X$ by Hodge symmetry. Then Popa-Schnell \cite{PS} show that 
\begin{equation}\label{eqPS}
a\ge n-p-\delta,
\end{equation}
where $\delta$ is the defect of the semismallness of the Albanese map. This fact is implicit in the proof of their result that $\codim \sV^{pq}_1(\cO_X)\ge |n-p-q|-\delta$, which follows from (\ref{eqPS}) together with part (d) of Proposition \ref{propIneqs} above.

(b) Consider the DGLA pair $(\Ar(\enmo(\bC_X)), \Ar(\bC_X))$ governing the deformation problem with cohomology constraints $
(\bR(X,1), \mathbf{1}, \ti\sV^i_k(\bC_X))
$
from Section \ref{localsystem}. When $X$ is the complement in $\bC^n$ of a central essential indecomposable hyperplane arrangement, the pair is formal because $X$ is formal. Moreoever, in this case, $a=n-1$ by \cite{EPY} and Proposition \ref{propIneqs} is proved in \cite{B-h}.
\end{rmk}


\begin{thebibliography}{DGMS}



\bibitem[Art68]{ar}M. Artin, \textsl{On the solutions of analytic equations.} Invent. Math. 5 (1968), 277-291. 

\bibitem[Art69]{ar1}M. Artin, \textsl{Algebraic approximation of structures over complete local rings.} Inst. Hautes \'Etudes Sci. Publ. Math. No. 36 (1969) 23-58.
 
\bibitem[Bud11]{B-h} N. Budur, \textsl{Complements and higher resonance varieties of hyperplane arrangements.} Math. Res. Lett. (5) 18  (2011), 859-873. \textsl{Erratum}, Math. Res. Lett. 21, no. 1 (2014). 

\bibitem[BW12]{bw}N. Budur, B. Wang, \textsl{Cohomology jump loci of quasi-projective varieties.} arXiv: 1211.3766. To appear in Ann. Sci. \'Ecole Norm. Sup.


\bibitem[DGMS75]{dgms} P. Deligne, P.  Griffiths, J. Morgan, D. Sullivan, \textsl{Real homotopy theory of K\"ahler manifolds}. Invent. Math. 29 (1975), no. 3, 245-274.

\bibitem[DPS09]{dps}A. Dimca, S. Papadima, A. Suciu, \textsl{Topology and geometry of cohomology jump loci.} Duke Math. J. 148 (2009), no. 3, 405-457.

\bibitem[DP12]{dp}A. Dimca, S. Papadima, \textsl{Nonabelian cohomology jump loci from an analytic viewpoint.}  arXiv:1206.3773. To appear in Comm. in Contemp. Math.

\bibitem[Eis95]{eisenbud}D. Eisenbud, \textsl{Commutative algebra. With a view toward algebraic geometry.} Graduate Texts in Mathematics, 150. Springer-Verlag, New York, 1995. 

\bibitem[EPY03]{EPY} D. Eisenbud, S. Popescu, S. Yuzvinsky,
\textsl{Hyperplane arrangement cohomology and monomials in the exterior algebra.} 
Trans. Amer. Math. Soc. 355 (2003), no. 11, 4365-4383. 

\bibitem[ESV93]{ESV} H. Esnault, V. Schechtman, E. Viehweg, \textsl{Cohomology of local systems of the complements of hyperplanes.} Invent. Math. 109 (1992), no. 3, 557-561; \textsl{Erratum}, ibid. 112 (1993), no. 2, 447.

\bibitem[FIM12]{FIM} D. Fiorenza, D. Iacono, E. Martinengo, \textsl{Differential graded Lie algebras controlling infinitesimal deformations of coherent sheaves.} J. Eur. Math. Soc. (JEMS) 14 (2012), no. 2, 521--540.

\bibitem[Ful98]{Fu} W. Fulton, \textsl{Intersection Theory. }Second edition. Ergebnisse der Mathematik und ihrer Grenzgebiete. Springer-Verlag, Berlin, 1998. xiv+470 pp.

\bibitem[GL87]{gl1}M. Green, R. Lazarsfeld, \textsl{Deformation theory, generic vanishing theorems, and some conjectures of Enriques, Catanese and Beauville.} Invent. Math. 90 (1987), no. 2, 389-407.

\bibitem[GL91]{gl2}M. Green, R. Lazarsfeld, \textsl{Higher obstructions to deforming cohomology groups of line bundles.} J. Amer. Math. Soc. 4 (1991), no. 1, 87-103.

\bibitem[GM88]{gm}W. Goldman, J. Millson, \textsl{Deformations of flat bundles over K\"ahler manifolds.} Inst. Hautes \'Etudes Sci. Publ. Math. No. 67 (1988), 43-96. 

\bibitem[GT09]{TiB} I. Grzegorczyk, M.  Teixidor i Bigas, \textsl{Brill-Noether theory for stable vector bundles.} {Moduli spaces and vector bundles,} 29Ð50, London Math. Soc. Lecture Note Ser., 359, Cambridge Univ. Press, Cambridge, 2009.

\bibitem[Har77]{h}R. Hartshorne, \textsl{Algebraic geometry.} Graduate Texts in Mathematics, No. 52. Springer-Verlag, New York-Heidelberg, 1977.

\bibitem[LP10]{LP} R. Lazarsfeld, M. Popa, \textsl{Derivative complex, BGG correspondence, and numerical inequalities for compact K\'{a}hler manifolds.} Invent. Math. 182 (2010), no. 3, 605-633. 


\bibitem[LO87]{lo}M. L\"ubke, C. Okonek, \textsl{Moduli spaces of simple bundles and Hermitian-Einstein connections.} Math. Ann. 276 (1987), no. 4, 663-674.


\bibitem[Lur]{Lu} J. Lurie, \textsl{Formal Moduli Problems.} \url{http://www.math.harvard.edu/~lurie/papers/DAG-X.pdf}


\bibitem[Man07]{M-a} M. Manetti, \textsl{Lie description of higher obstructions to deforming submanifolds.} Ann. Sc. Norm. Super. Pisa Cl. Sci. (5) 6 (2007), no. 4, 631--659.

\bibitem[Man09]{M} M. Manetti, \textsl{Differential graded Lie algebras and formal deformation theory.} {Algebraic geometry -- Seattle 2005.} Part 2, 785--810, Proc. Sympos. Pure Math., 80, Part 2, Amer. Math. Soc., Providence, RI, 2009.

\bibitem[Mar09]{ma}E. Martinengo, \textsl{Local structure of Brill-Noether strata in the moduli space of flat stable bundles.} Rend. Semin. Mat. Univ. Padova 121 (2009), 259-280.

\bibitem[Mar12]{Mart} E. Martinengo, \textsl{Infinitesimal deformations of Hitchin pairs and Hitchin map.} Internat. J. Math. 23 (2012), no. 7, 125-153.


\bibitem[Nad88]{n}A. Nadel, \textsl{Singularities and Kodaira dimension of the moduli space of flat Hermitian-Yang-Mills connections.} Compositio Math. 67 (1988), no. 2, 121-128.


\bibitem[PoSc11]{PS} M. Popa, C. Schnell,  \textsl{Generic vanishing theory via mixed Hodge modules.} Forum of Mathematics, Sigma 1 (2013), 1-60.

\bibitem[Pri11]{Pr} J.P. Pridham, \textsl{Unifying derived deformation theories. }Adv. Math. 224 (2010), no. 3, 772-826. \textsl{Corrigendum.} Adv. Math. 228 (2011), no. 4, 2554-2556.

\bibitem[Rob98]{R} P. Roberts, \textsl{Multiplicities and Chern classes in local algebra}, CTM 133, Cambridge University Press, 1998.

\bibitem[Sim92]{s1}C. Simpson, \textsl{Higgs bundles and local systems.} Inst. Hautes \'Etudes Sci. Publ. Math. No. 75 (1992), 5-95.


\bibitem[Sim94]{s3}C. Simpson, \textsl{Moduli of representations of the fundamental group of a smooth variety. II.} Inst. Hautes \'Etudes Sci. Publ. Math. No. 80 (1994), 5-79.



\bibitem[UY86]{uy} K. Uhlenbeck, S.-T. Yau, \textsl{On the existence of Hermitian-Yang-Mills connections in stable vector bundles.} Comm. Pure Appl. Math. 39 (1986), no. S, suppl., S257-S293. 

\bibitem[Wan12]{w} B. Wang, \textsl{Cohomology jump loci in the moduli spaces of vector bundles.} arXiv:1210.1487.

\bibitem[Wan13a]{wb} B. Wang,  \textsl{Cohomology jump loci of compact K\"{a}hler manifolds.} arXiv:1303.6937v1.

\bibitem[Wan13b]{w-ex} B. Wang, \textsl{Examples of topological spaces with arbitrary cohomology jump loci.} arXiv:1304.0239.

\end{thebibliography}
\end{document}